\documentclass[a4paper,10pt]{article}
\usepackage{geometry}
\usepackage[T1]{fontenc}
\usepackage[utf8]{inputenc}
\usepackage{lmodern}
\usepackage{microtype}
\usepackage{amsmath,amssymb,amsthm,mathtools}
\usepackage{enumitem}
\usepackage{hyperref}
\usepackage{mathrsfs}
\allowdisplaybreaks

\hypersetup{colorlinks=true,linkcolor=blue,citecolor=blue,urlcolor=blue}

\newtheorem{theorem}{Theorem}[section]  
\newtheorem{lemma}[theorem]{Lemma}      
\newtheorem{definition}[theorem]{Definition}
\newtheorem{proposition}[theorem]{Proposition}
\newtheorem{corollary}[theorem]{Corollary}
\newtheorem{case}{Case}
\newtheorem{claim}{Claim}
\theoremstyle{definition}
\newtheorem{problem}{Problem}

\newtheorem{conjecture}{Conjecture}

\newcommand{\NC}{N_C}
\newcommand{\dstar}{d^{\ast}}

\title{Longest cycles and Dirac-type results in highly connected graphs}
\author{Jie Ma\thanks{School of Mathematical Sciences, University of Science and Technology of China, Hefei, Anhui 230026, China, and Yau Mathematical Sciences Center, Tsinghua University, Beijing 100084, China. Research supported by National Key Research and Development Program of China 2023YFA1010201 and National Natural
Science Foundation of China grant 12125106. Email: \url{jiema@ustc.edu.cn}.}
\and Bo Ning\thanks{College of Computer Science, Nankai University, Tianjin 300350, P.R. China. Research supported by National Natural Science Foundation of China grant 12371350 and the Fundamental Research Funds for the Central Universities, Nankai University (No. 63263259). Email: \url{bo.ning@nankai.edu.cn}.}
\and Ziyuan Zhao\thanks{School of Mathematical Sciences, University of Science and Technology of China, Hefei, Anhui 230026, China. Research supported by Quantum Science and Technology-National Science and Technology Major Project 2021ZD0302902. Email: \url{ zyzhao2024@mail.ustc.edu.cn}.}}

\begin{document}
\maketitle

\begin{abstract}
A classical theorem of Nash-Williams states that if $G$ is a $2$-connected graph on $n$ vertices with minimum degree at least $(n+2)/3$, then for every longest cycle $C$ of $G$, the graph $G-V(C)$ is edgeless. Motivated by a higher-connectivity analogue, Bondy conjectured in 1980 that if $G$ is a $k$-connected graph on $n$ vertices with minimum degree at least $(n+k(k-1))/(k+1)$, then for every longest cycle $C$ of $G$, every path in $G-V(C)$ has at most $k-1$ vertices. This conjecture is known for $k\le 3$ and remains open for all $k\ge 4$.

In this paper, we prove Bondy's conjecture for all sufficiently large graphs. The key ingredient is a new Dirac-type theorem that gives a lower bound on the length of a longest cycle in a $k$-connected graph, which also yields a partial solution to a conjecture of Jung from 1990. Along the way, we develop several new tools, including a DFS lemma and an average-degree analogue of the Bondy--Jackson theorem. We conclude with a discussion of related problems and a counterexample to a conjecture of Voss from 1991.
\end{abstract}

\section{Introduction}

The study of longest cycles is a central topic in graph theory.
A fundamental theme in this area is the interplay between the structural parameters of a graph and the properties of its longest cycles.
A cornerstone result in this direction is Dirac's theorem~\cite{Dirac1952}, established in 1952, which asserts that every graph $G$ on $n\ge 3$ vertices with minimum degree $\delta(G)\ge n/2$ contains a {\it Hamilton cycle}, that is, a cycle passing through every vertex of $G$. 
Over the past 70 years, Dirac's theorem has served as a driving force behind the development of long-cycle theory, leading to numerous extensions and refinements; see the surveys~\cite{G2014,KO2014,LiSurvey2013}, and, for example, the recent works \cite{ChenShan2019,DCS2024,FGSS2022,KLM2024-1,KLM2024-2,KrivelevichLeeSudakov2014,KrivelevichLeeSudakov2017,LiNing2021,MaNing2020,MaYuan2024,NingYuan2025,ZhuEtAl2023}.

Another landmark result in the study of longest cycles is the dominating cycle theorem of Nash-Williams~\cite{NashWilliams1971}.
It states that if $G$ is a $2$-connected graph on $n\ge 3$ vertices with minimum degree $\delta(G)\ge (n+2)/3$, then every longest cycle $C$ of $G$ is {\it dominating}, that is, the graph $G-V(C)$ consists of isolated vertices.
This theorem has proved to be a powerful tool in Hamiltonian graph theory and has found sustained influence, including the study of Hamilton cycles and pancyclicity~\cite{AGK2021,BrandtFaudreeGoddard1998,BBVL1989,Fassbender1991}, edge-disjoint Hamilton cycles~\cite{ChristofidesKuhnOsthus2012,CsabaKuhnLoOsthusTre2016,FriezeKrivelevich2005,NashWilliams1971}, and Hamilton cycles in regular graphs \cite{BollobasHobbs1978,ErdosHobbs1978,KLOS2016,Woodall1978}.

Motivated by a higher-connectivity extension of these two theorems, Bondy~\cite{Bondy1980Report} 
put forward the following well-known conjecture in 1980.\footnote{Bondy's original conjecture was stated in a more general Ore-type form. Throughout this paper, we work with the minimum-degree formulation above, which can be viewed as a natural specialization of the original conjecture and has attracted considerable attention in the folklore of the area.}

\begin{conjecture}[Bondy~\cite{Bondy1980Report}]\label{conj:bondy}
Let $k\geq 1$ be an integer, and let $G$ be a $k$-connected graph on $n$ vertices. If
$\delta(G)\ge \frac{n+k(k-1)}{k+1},$
then for every longest cycle $C$ of $G$, every path in $G-V(C)$ has at most $k-1$ vertices.
\end{conjecture}

If true, this bound is sharp, by considering the following example:
for integers $t\geq k\ge 1$, let $G$ be the graph obtained from $k+1$ pairwise disjoint cliques of order $t$ by joining every vertex of a clique $K_k$ to every vertex of each of these cliques.
It can be verified that this graph $G$ is $k$-connected and has minimum degree $\frac{n+k(k-1)-1}{k+1}$ and for each longest cycle $C$ of $G$, $G-V(C)$ contains a path of $t\geq k$ vertices. 
We note that the cases $k=1$ and $k=2$ of Conjecture~\ref{conj:bondy} are precisely Dirac's
theorem and Nash-Williams' theorem, respectively.  
In \cite{Bondy1980Report}, Bondy proved a weaker version, by imposing strong assumptions on the connectivity of $G-V(C)$. 
Zhang~\cite{Zhang1988} verified the conjecture for claw-free graphs (i.e., induced $K_{1,3}$-free graphs).
A variant of this conjecture can be found in \cite{OzekiTsugakiYamashita2009}.
Conjecture~\ref{conj:bondy} was proved for $k\leq 3$ and remains open in general for $k\ge 4$ (see, for example, \cite{OzekiTsugakiYamashita2009}). 

In this paper, we confirm Conjecture~\ref{conj:bondy} for all sufficiently large graphs. Our main result is the following. 

\begin{theorem}\label{thm:main-bondy}
For every $k\geq 1$, there exists $n_k$ such that for any $k$-connected graph $G$ on $n\geq n_k$ vertices and  any longest cycle $C$ in $G$, if $\delta(G)\ge \frac{n+k(k-1)}{k+1}$, then every path in $G-V(C)$ has at most $k-1$ vertices.
\end{theorem}

\noindent Our proof yields the explicit bound $n_k=5k^2+7k$. Although we have not attempted to optimize the constant factors, the dependence on $k$ is quadratic with a relatively small leading coefficient.
The main ingredient in our proof is a new Dirac-type theorem on long cycles in $k$-connected graphs under natural additional assumptions.

Recall Dirac's classical long cycle theorem~\cite{Dirac1952}, which states that every $2$-connected graph on $n$ vertices contains a cycle of length at least $\min\{2\delta,n\}$. 
One might hope that the factor $2$ in the lower bound can be replaced by a quantity proportional to the connectivity. However, this is impossible in general, for instance, as shown by the complete bipartite graph $K_{k,n-k}$.
Motivated by this observation, Jung~\cite{Jung2001} conjectured that if $G$ is a $k$-connected graph, $C$ is a longest cycle of $G$, and $G-V(C)$ contains a path with at least $k-1$ vertices, then
$|C|\ge k(\delta-k+2).$
This has been verified for $k\le 6$; see~\cite{VumarJung2005}.

The following result plays an important role in the proof of Theorem~\ref{thm:main-bondy}. It confirms Jung's conjecture under the additional assumption $\delta(G)\ge 6k$. (We make no attempt to optimize the constant $6$.)

\begin{theorem}\label{thm:main-jung}
For $k\ge 2$, let $G$ be a $k$-connected graph with minimum degree $\delta\ge 6k$, and let $C$ be a longest cycle in $G$.
If $G-V(C)$ contains a path on at least $k-1$ vertices, then $|C|\ge k(\delta-k+2).$
\end{theorem}

To prove Theorems~\ref{thm:main-bondy} and~\ref{thm:main-jung}, we consider their ``local'' versions, namely Theorems~\ref{thm:local-bondy} and~\ref{thm:local-jung}, respectively. 
Roughly speaking, these local formulations are more amenable to analysis, as they admit a number of useful local operations. To handle the challenges arising from the local structure, we develop several new tools, including a DFS-tree lemma and an average-degree analogue of the Bondy--Jackson theorem~\cite{BondyJackson1985}.
We refer the reader to Section~\ref{Sec:2} for a more detailed discussion.

The paper is organized as follows.
In Section~\ref{Sec:2}, we reduce Theorems~\ref{thm:main-bondy} and~\ref{thm:main-jung} to their local statements Theorems~\ref{thm:local-bondy} and~\ref{thm:local-jung}, and outline the proof strategy.
In Section~\ref{sec:pre}, we collect the notation and auxiliary lemmas used throughout the paper.
The two main technical ingredients are established in Sections~\ref{Sec:2.5} and~\ref{Sec:3}: an average-degree version of the Bondy--Jackson theorem and a DFS lemma, respectively.
These ingredients are combined in Section~\ref{Sec:4} to prove Theorems~\ref{thm:local-bondy} and~\ref{thm:local-jung}.
Finally, in Section~\ref{sec:conclude}, we conclude with a discussion of related conjectures and open problems motivated by this work, including a positive result on a conjecture of Fan, a counterexample to a conjecture of Voss, and two further problems.

\section{Reduction and proof sketch}\label{Sec:2}
In this section, we show how Theorems~\ref{thm:main-bondy} and~\ref{thm:main-jung} follow from their local counterparts, Theorems~\ref{thm:local-bondy} and~\ref{thm:local-jung}, respectively, and then provide an outline of the proofs of the latter two theorems.

We begin with some notation.
Let $G$ be a graph, $C$ a cycle of $G$, and $H$ a component of $G-V(C)$.
We say that \(C\) is \emph{locally longest with respect to \(H\)}
if, for every two distinct vertices \(x,y\in V(C)\), every path in \(G\) whose endpoints are \(x\) and \(y\) and whose internal vertices
all lie in \(V(H)\) has length at most the minimum of the lengths of the two subpaths of \(C\) with endpoints \(x\) and \(y\). The component $H$ is \emph{locally $k$-connected to $C$} if, for every $x\in V(H)$, there exist $k$ paths from $x$ to $k$ distinct vertices of $C$, any two of which intersect only in $x$.

\begin{definition}
Let $G$ be a connected graph, $C$ a non-Hamiltonian cycle of $G$, and $H$ a component of $G-V(C)$. For an integer $k\ge 2$, we say that $(G,C,H)$ is a {\bf $k$-triple system} if $C$ is locally longest with respect to $H$ in $G$, and $H$ is locally $k$-connected to $C$ in $G$.
\end{definition}

We now state the local theorems. Let $p(H)$ denote the number of vertices of a longest path in $H$.
Thus, $p(H)$ is an order parameter, not an edge-length parameter.

\begin{theorem}\label{thm:local-bondy}
For $k\ge 2$, let $(G,C,H)$ be a $k$-triple system with $|V(G)|=n$. If
$\delta_G(H)\ge \frac{n+k(k-1)}{k+1}$
and
$n\ge 5k^2+7k$,
then $p(H)\leq k-1$.
\end{theorem}

\begin{theorem}\label{thm:local-jung}
For $k\ge 2$, let $(G,C,H)$ be a $k$-triple system. If $\delta_G(H)\ge 6k$ and $p(H)\ge k-1$, then
$|C|\ge k\bigl(\delta_G(H)-k+2\bigr).$
\end{theorem}

\medskip

\noindent {\bf \underline{Proof of the reductions.}}
We now deduce Theorems~\ref{thm:main-bondy} and~\ref{thm:main-jung} from Theorems~\ref{thm:local-bondy} and~\ref{thm:local-jung}.

\begin{proof}[Proof of Theorem~\ref{thm:main-bondy} (assuming Theorem~\ref{thm:local-bondy})]
Let $n_k=5k^2+7k.$
The case $k=1$ follows from Dirac's Theorem.
Assume $k\geq 2$. Let $G$ be a $k$-connected graph on $n\ge n_k$ vertices with
$\delta(G)\ge \frac{n+k(k-1)}{k+1},$
and let $C$ be a longest cycle of $G$.
If $C$ is Hamiltonian, then the conclusion is immediate.
Otherwise, for every component $H$ of $G-V(C)$,
Menger's theorem and the maximality of $C$ imply that $(G,C,H)$ is a
$k$-triple system. Hence Theorem~\ref{thm:local-bondy} implies that $p(H)\le k-1$, as desired.
\end{proof}

\begin{proof}[Proof of Theorem~\ref{thm:main-jung} (assuming Theorem~\ref{thm:local-jung})]
Let $G$ be a $k$-connected graph with minimum degree $\delta\ge 6k$, and let $C$ be a longest cycle of $G$. Suppose that $G-V(C)$ contains a path with
at least $k-1$ vertices, and let $H$ be the component containing such a
path. Then $(G,C,H)$ is a $k$-triple system, and clearly
$\delta_G(H)\ge \delta$. 
Therefore, Theorem~\ref{thm:local-jung} gives
$|C|\ge k\bigl(\delta_G(H)-k+2\bigr)\ge k(\delta-k+2)$, as desired.
\end{proof}

\medskip
\noindent {\bf \underline{Proof sketch.}}
The central technical part of the paper is the proof of Theorem~\ref{thm:local-jung}, from which Theorem~\ref{thm:local-bondy} follows (see Subsection~\ref{sec:final proof}). 
We thus concentrate on Theorem~\ref{thm:local-jung}.

Let $(G,C,H)$ be a $k$-triple system with $\delta:=\delta_G(H)\ge 6k$ and $p:=p(H)\ge k-1$.
To obtain a lower bound on the length of the cycle $C$, the proof employs two principal strategies.
The first is Fan's attachment argument (Lemma~\ref{lem:23}), which turns long paths in $H$ between vertices attached to $C$ into lower bounds on the corresponding segments of $C$.
The second is based on Equation~\eqref{eq:41-counting}, which gives $|C|\ge e(U,C)+|N_C(U)|$, where the key is to find a subset $U\subseteq V(H)$ that is suitably linked to $C$.
Here $e(U,C)$ denotes the number of edges between $U$ and $V(C)$, while $N_C(U)$ consists of vertices in $C$ incident to $U$.

Let $c(H)$ denote the length of a longest cycle in $H$. 
The proof of Theorem~\ref{thm:local-jung} is divided into three regimes according to $p$ and $c(H)$,
with each regime employing one of the above strategies; see Table~\ref{tab:roadmap}.

When either $p$ or $c(H)$ is small,
we apply Equation~\eqref{eq:41-counting} and use a DFS-tree lemma (Lemma~\ref{lem:31}) to select a suitable subset $U\subseteq V(H)$. 
The crucial property is that the total degree of $U$ in $H$ can be bounded in terms of $p$ and $c(H)$, which in turn yields a strong lower bound on $e(U,C)$.
In the remaining regime, we instead employ Fan's attachment argument.
To make this work, we first obtain estimates on the lengths of paths in $H$ via average-degree versions of the Bondy--Jackson theorem (Theorems~\ref{thm:24} and~\ref{thm:25}). 
Combining these estimates with Fan's attachment argument then gives the desired lower bound on $|C|$.

\begin{table}[htbp]
\centering
\caption{Roadmap for the proof of Theorem~\ref{thm:local-jung}.}
\begingroup
\renewcommand{\arraystretch}{1.18}
\small
\begin{tabular}{@{}p{0.29\textwidth}p{0.39\textwidth}p{0.24\textwidth}@{}}
\hline
\textbf{Case Conditions} &
\textbf{Tools} &
\textbf{Proof Location} \\
\hline
$k-1\le p\le \delta-k+1$ &
\eqref{eq:41-counting} $+$ DFS lemma & Subsection~\ref{sec:small p}\\[1mm]

$p\ge \delta-k+2$ and $c(H)\ge 5k/2$ &
attachment argument $+$ average Bondy--Jackson &
Subsection~\ref{sec:large c}\\[1mm]

$p\ge \delta-k+2$ and $c(H)<5k/2$ &
\eqref{eq:41-counting} $+$ DFS lemma &
Subsection~\ref{sec:final proof} \\
\hline
\end{tabular}
\endgroup
\label{tab:roadmap}
\end{table}

\section{Preliminaries}\label{sec:pre}
We use the following notation throughout.
For a vertex set $A\subseteq V(G)$ or a subgraph $A$ of a graph $G$, we write $G-A$ for the
graph obtained from $G$ by deleting all vertices in $A$. For vertex sets
or subgraphs $A,B$ of $G$, let $N_A(B)$ denote the set of vertices in $A$
but not in $B$ that have a neighbor in $B$. When $B=\{b\}$, write
$N_A(b):=N_A(\{b\})$ and define $d_A(b):=|N_A(b)|$.
If $A$ and $B$ are disjoint, we write $e(A,B)$ for the number of edges of $G$ with one endpoint in $A$ and the other in $B$.
We use $K_n$ to denote a complete graph on $n$ vertices. Let $G$ and $H$ be two graphs. For a positive integer $t$,
let $tH$ denote the disjoint union of $t$ copies of $H$.  We use $G\cup H$ and $G\vee H$ to denote the disjoint union and join of $G$ and $H$, respectively. The symbol $\delta(G)$ denotes the minimum degree of $G$.

For $x,y\in V(G)$ and a subgraph $H$ of $G$, an $(x,y)$-path is a path
with endpoints $x$ and $y$, and an $(x,H,y)$-path is an $(x,y)$-path
whose internal vertices all lie in $V(H)$; in particular, the edge $xy$,
when present, is an $(x,H,y)$-path.
Let $\delta_G(H):=\min\{d_G(v): v\in V(H)\}.$
We write $\dstar_H(x,y)$ for the length of a longest $(x,H,y)$-path if such a path exists, and set $\dstar_H(x,y)=0$ otherwise.
For a path or cycle $Q$, its length $|Q|$ is the number of edges of $Q$. The circumference of a graph $H$, denoted by $c(H)$, is the length of a longest cycle of $H$; if $H$ is a forest we set $c(H)=2$. 
Let $p(H)$ denote the number of vertices of a longest path in $H$.

For vertices $a,b$ on a path $P$, write $P[a,b]$ for the subpath of $P$ with endpoints $a$ and $b$.
Let $C$ be a cycle in $G$.
Two distinct vertices $x,y\in V(C)$ are called \emph{strongly attached} to a subgraph $H$ in $G$ if there exist distinct vertices $z_1,z_2\in V(H)$ such that $xz_1,yz_2\in E(G)$. A subset $T=\{x_1,x_2,\dots,x_t\}\subseteq V(C)$ is called a \emph{strong attachment} of $H$ to $C$ if the vertices $x_i$ appear on $C$ in cyclic order and each consecutive pair $x_i,x_{i+1}$ is strongly attached, where indices are always taken modulo $t$.
It is \emph{maximal} if it is not properly contained in another strong
attachment, and \emph{maximum} if it has maximum cardinality.
For a set $X=\{x_1,x_2,\dots,x_t\}\subseteq V(C)$ with $t\ge 2$ listed in cyclic order on $C$, we write $C[x_i,x_{i+1}]$ for the segment of $C$ from $x_i$ to $x_{i+1}$ according to this cyclic order, where indices are taken modulo $t$. Each $C[x_i,x_{i+1}]$ is called an $(x_i,x_{i+1})$-segment of $C$. These $X$-segments are edge-disjoint and partition the edges of $C$; in particular, $|C|=\sum_{i=1}^t |C[x_i,x_{i+1}]|$.

We now prove and collect a sequence of lemmas that will be used in the subsequent proofs. 
The first lemma guarantees the existence of a long subpath.

\begin{lemma}\label{lem:21}
Let $P$ be a path and let $A,B\subseteq V(P)$ be two non-empty subsets, possibly intersecting. Then there exist vertices $a\in A$ and $b\in B$ such that the subpath $P[a,b]$ has length at least
$\frac{|A|+|B|}{2}-1.$
\end{lemma}

\begin{proof}
Order the vertices of $P$ from one end to the other. Let $a_1$ and $a_2$ be the first and the last vertices of $A$ on $P$, respectively, and let $b_1$ and $b_2$ be the first and the last vertices of $B$ on $P$, respectively.

We first observe that $|P[a_1,b_2]|+|P[a_2,b_1]|\ge |P[a_1,a_2]|+|P[b_1,b_2]|$. Indeed, consider any edge $e$ of $P$. If $e$ is counted twice by $|P[a_1,a_2]|+|P[b_1,b_2]|$, then $e$ separates $a_1$ from $a_2$ and also separates $b_1$ from $b_2$, so it is counted twice by $|P[a_1,b_2]|+|P[a_2,b_1]|$. If $e$ is counted exactly once by $|P[a_1,a_2]|+|P[b_1,b_2]|$, say by $|P[a_1,a_2]|$, then $e$ separates $a_1$ from $a_2$, while $b_1$ and $b_2$ lie on the same side of $e$; hence $e$ separates either $a_1$ from $b_2$ or $a_2$ from $b_1$, and so it is counted at least once by $|P[a_1,b_2]|+|P[a_2,b_1]|$. Thus the observation follows.

Since $A\subseteq V(P[a_1,a_2])$ and $B\subseteq V(P[b_1,b_2])$, we have $|P[a_1,a_2]|\ge |A|-1$ and $|P[b_1,b_2]|\ge |B|-1$. The observation thus gives $|P[a_1,b_2]|+|P[a_2,b_1]|\ge |A|+|B|-2$. It follows that at least one of $|P[a_1,b_2]|$ and $|P[a_2,b_1]|$ is at least $(|A|+|B|-2)/2=\frac{|A|+|B|}{2}-1$. Since $a_1,a_2\in A$ and $b_1,b_2\in B$, this gives the desired vertices $a\in A$ and $b\in B$.
\end{proof}

A set of edges is called \emph{independent} if no two of its edges share an endpoint. 
The following two lemmas are taken from Fan's treatment of the attachment argument \cite{Fan1990}.

\begin{lemma}[\rm {Fan \cite[Proposition~1]{Fan1990}}]\label{lem:22}
Let $C$ be a cycle in a graph $G$ and let $H$ be a component of $G-C$. If $H$ is locally $k$-connected to $C$ in $G$, then there are at least $\min\{k,|V(H)|\}$ pairwise independent edges between $C$ and $H$.
\end{lemma}

The next lemma is obtained by combining Lemma~1(a), (d), and (e) with Lemma~3(a) of Fan~\cite{Fan1990}.
\begin{lemma}[Fan \cite{Fan1990}]\label{lem:23}
Let $C$ be a cycle in a graph $G$ and let $H$ be a component of $G-C$. 
Let $T=\{u_1,u_2,\dots,u_t\}$ be a maximal strong attachment of $H$ to $C$ such that the vertices $u_i$ appear on $C$ in cyclic order, where $t\ge 2$. Let $h:=|V(H)|$.
Then the following hold:
\begin{enumerate}[label=(\roman*),leftmargin=2.3em]
\item Every vertex in $N_C(H)\setminus T$ has exactly one neighbor in $H$.
\item If $C$ is locally longest with respect to $H$ in $G$, then
$|C|\ge \sum_{i=1}^t \dstar_H(u_i,u_{i+1}) + 2|N_C(H)\setminus T|.$
\item Let $D(T)=\{v\in T : d_H(v)\ge 2\}$. Then we have:
\begin{itemize}
\item If $H$ is locally $k$-connected to $C$ in $G$ and $|V(C)|\ge k$, then
$t\ge \min\{k, h+|D(T)|\}.$
\item If $H$ is $2$-connected and has average degree $\frac{1}{|V(H)|}\sum_{v\in V(H)}d_G(v)=r$, then for every strongly attached pair $\{u_i,u_{i+1}\}$,
$\dstar_H(u_i,u_{i+1})\ge r+2-t-\frac{|\NC(H)\setminus T|}{h}.$
\end{itemize}
\end{enumerate}
\end{lemma}

A graph is \emph{Hamiltonian-connected} if any two of its vertices are joined by a Hamilton path. 
The following theorem of Ore~\cite{Ore1963} provides a sufficient condition for Hamiltonian-connectivity.

\begin{theorem}[Ore~\cite{Ore1963}]\label{thm:ore-hc}
Let $G$ be a graph on $n\ge 4$ vertices. If $\delta(G)\ge 3$ and
$e(G)\ge \binom{n-1}{2}+2,$
then $G$ is Hamiltonian-connected except when $n=6$ and $G=K_3\vee 3K_1$.
\end{theorem}

The following corollary follows immediately from Theorem~\ref{thm:ore-hc}.

\begin{corollary}\label{cor:ore-hc}
Let $G$ be a graph on $n\ge 5$ vertices. If
$e(G)\ge \binom{n-1}{2}+2,$
then $G$ is Hamiltonian-connected, except in the following two cases:
(i) $n=6$ and $G=K_3\vee 3K_1$;
(ii) $G=K_2\vee (K_1\cup K_{n-3})$.
\end{corollary}

\begin{proof}
Suppose that $G$ is not Hamiltonian-connected. If $\delta(G)\ge 3$, then Theorem~\ref{thm:ore-hc} yields Case~(i). So we assume $\delta(G)\le 2$.
Thus, $G$ contains a vertex $x$ of degree at most two.
Then $\binom{n-1}{2}+2\leq e(G)\leq 2+e(G-\{x\})$, 
implying that $x$ has degree exactly two, and $G-\{x\}$ is a clique of size $n-1$. This yields Case~(ii).
\end{proof}

Fan~\cite{Fan1990} proved the following average-degree version of the Erd\H{o}s--Gallai theorem.

\begin{theorem}[Fan~\cite{Fan1990}]\label{thm:fan-average}
Let $G$ be a $2$-connected graph on $n\ge 3$ vertices and let $x,y\in V(G)$ be distinct. Then there exists an $(x,y)$-path in $G$ of length at least
$\sum\limits_{v\in V(G)\setminus\{x,y\}} d_G(v)/(n-2)
=\frac{2e(G)-d_G(x)-d_G(y)}{n-2}.$
\end{theorem}

We often use the following variant of the above theorem of Fan.

\begin{lemma}\label{lem:24}
Let $C$ be a cycle in a graph $G$ and let $H$ be a 2-connected component of $G-C$. For any $x,y\in V(C)$, if $\{x,y\}$ is a strongly attached pair, then
$\dstar_H(x,y)\ge d(H)+\frac{e(\{x,y\},H)}{|V(H)|},$
where $e(\{x,y\},H)$ denotes the number of edges between $\{x,y\}$ and $H$, and
$d(H):=\sum_{v\in V(H)} |N_G(v)\cap V(H)|/|V(H)|.$
\end{lemma}

\begin{proof}
Let $H'$ be the induced subgraph of $G$ on $V(H)\cup\{x,y\}$, adding the edge $xy$ if it is not already present. Since $H$ is $2$-connected, so is $H'$. Then
$\sum_{v\in V(H')\setminus\{x,y\}} d_{H'}(v)
=\sum_{v\in V(H)} d_H(v)+e(\{x,y\},H).$
Applying Theorem~\ref{thm:fan-average} to $H'$ with distinct vertices $x$ and $y$ gives
$\dstar_H(x,y)
\ge \sum\limits_{v\in V(H')\setminus\{x,y\}} d_{H'}(v)/{|V(H)|}
= d(H)+\frac{e(\{x,y\},H)}{|V(H)|}.$
\end{proof}

We conclude this section with the following theorem of Bondy and Jackson, extending the classical Erd\H{o}s--Gallai theorem by permitting a small number of exceptions to the minimum-degree requirement.
In the next section, we prove an average-degree analogue of this result.

\begin{theorem}[Bondy and Jackson~\cite{BondyJackson1985}]\label{thm:bondy-jackson}
Let $G$ be a $2$-connected graph on $n\ge 4$ vertices. Let $x,y$ be distinct vertices and let $z\in V(G)\setminus\{x,y\}$. If every vertex $w\in V(G)\setminus\{x,y,z\}$ satisfies $d_G(w)\ge k$, then there exists an $(x,y)$-path of length at least $k$.
\end{theorem}

\section{An average-degree version of the Bondy--Jackson theorem}\label{Sec:2.5}

In this section, we establish an average-degree version of Theorem~\ref{thm:bondy-jackson},
which will serve as a central tool in the proof of Theorem~\ref{thm:local-jung}.
Our proof follows Fan's induction argument for Theorem~\ref{thm:fan-average}, 
supplemented by a more involved structural analysis.

\begin{theorem}\label{thm:24}
Let $G$ be a $2$-connected graph on $n\ge 4$ vertices. Let $x,y,z\in V(G)$ be distinct. Then there exists an $(x,y)$-path of length at least
$r:=\frac{2e(G)-d_G(x)-d_G(y)-d_G(z)}{n-3}.$
\end{theorem}

\begin{proof}
We may add the edge $xy$ (if it is not in $E(G)$), since this does not change $r$. An $xy$-cycle means a cycle that contains the edge $xy$. Let $C$ be a longest $xy$-cycle, and put $c:=|C|$. We prove by induction on $n$ that $c\ge r+1$ (one can see that deleting $xy$ from such a cycle gives the required $(x,y)$-path). Assume, to the contrary, that
\begin{equation}\label{eq:24-assume}
 c<r+1.
\end{equation}

If $n=4$, then necessarily $c=3$. Writing $V(G)=\{x,y,z,u\}$ and taking a triangle $xyux$, we must have $N_G(z)=\{x,y\}$, since otherwise $c=4$. Thus, $G=K_4-e$, and then $r+1=3=c$. Hence, we may assume $n\ge 5$.
We may also assume $c\ge 4$: otherwise $G-\{x,y\}$ has no edges by $2$-connectedness, so $G=K_2\vee (n-2)K_1$, whence $r=2$ and $c=3$, contradicting~\eqref{eq:24-assume}.
Since $G$ has maximum degree at most $n-1$, it follows that $r\leq n-1$ and $c<n$.
Hence, $G-V(C)$ is not empty.

In the following, we call an $xy$-cycle \emph{good} if it has length at least $r+1$.
We divide the remaining part of the proof into three cases.
\begin{case}\label{Case:1}
There exists a component $H$ of $G-V(C)$ not containing $z$ such that $H$ is either $2$-connected, or an edge, or a single vertex.    
\end{case}

\noindent Let $h:=|V(H)|$ and let
$r_1:=\frac{\sum_{v\in V(H)} d_G(v)}{h}.$
Set $G_1:=G-H$. Then $G_1$ is $2$-connected, $|V(G_1)|\ge c\ge 4$, and
\begin{align*}
\sum_{v\in V(G_1)\setminus\{x,y,z\}} d_{G_1}(v)
&=\sum_{v\in V(G)\setminus\{x,y,z\}} d_G(v)-\sum_{v\in V(H)} d_G(v)-e(H,C-\{x,y,z\}) \\
&=(n-3)r-r_1h-e(H,C-\{x,y,z\}).
\end{align*}
If the right-hand side is at least $(n-h-3)r$, then the induction hypothesis applied to $G_1$ yields a good $xy$-cycle in $G_1$, and hence in $G$. Therefore, we may assume that
$(n-3)r-r_1h-e(H,C-\{x,y,z\})<(n-h-3)r,$
that is,
\begin{equation}\label{eq:24-case1-r1}
 r_1h>rh-e(H,C-\{x,y,z\}).
\end{equation}

Suppose first that $H$ is $2$-connected or an edge. Then $h\geq 2$. 
Let $T=\{u_1,u_2,\dots,u_t\}$ be a maximum strong attachment of $H$ to $C$, in cyclic order on $C$. Since $G$ is $2$-connected, we have $t\ge 2$ by Lemma \ref{lem:22}. Let $S:=\NC(H)\setminus T$ and $s:=|S|$. Choose an index $i$ such that the segment $C[u_i,u_{i+1}]$ does not contain the edge $xy$, and set
$G_i:=G[V(H)\cup \{u_i,u_{i+1}\}] + u_iu_{i+1}.$
Then $G_i$ is $2$-connected and $|V(G_i)|=h+2\ge 4$. By~\eqref{eq:24-case1-r1},
\begin{equation}\label{eq:temp-case1}
\sum_{v\in V(G_i)\setminus\{u_i,u_{i+1}\}} d_{G_i}(v)
= r_1h-e(H,C-\{u_i,u_{i+1}\})
> rh-e(H,C-\{x,y,z\})-e(H,C-\{u_i,u_{i+1}\}).
\end{equation}
Set $\theta:=|\{x,y,z\}\cap T|$ and $\sigma:=|\{x,y,z\}\cap S|$. Observe that each vertex of $T$ has at most $h$ neighbors in $H$, and each vertex of $S$ has exactly one neighbor in $H$ by Lemma~\ref{lem:23}(i).  This implies
$e(H,C-\{x,y,z\})\le (t-\theta)h+s-\sigma,$ and similarly, 
$e(H,C-\{u_i,u_{i+1}\})\le (t-2)h+s.$
Hence,
\[
\frac{\sum\limits_{v\in V(G_i)\setminus\{u_i,u_{i+1}\}} d_{G_i}(v)}{h}
> r-2t+2+\theta-\frac{2s-\sigma}{h}.
\]
By Theorem~\ref{thm:fan-average}, $G_i$ contains a $(u_i,u_{i+1})$-path of length greater than the right-hand side. 
This immediately gives a
$(u_i,H,u_{i+1})$-path in $G$ of length greater than the same quantity. Hence, whenever $C[u_i,u_{i+1}]$ avoids $xy$,
\begin{equation}\label{eq:24-case1-path}
 \dstar_H(u_i,u_{i+1})> r-2t+2+\theta-\frac{2s-\sigma}{h}.
\end{equation}

Let $p$ be the index such that the segment $C[u_p,u_{p+1}]$ contains the edge $xy$, and set
$\theta':=|\{x,y\}\cap T|=|\{x,y\}\cap\{u_p,u_{p+1}\}|.$
Put $S_p:=S\cap V(C[u_p,u_{p+1}])$ and $\phi:=|\{x,y\}\cap S_p|$. Then $\sigma\ge \phi$. 
The vertices of $\{u_p,u_{p+1}\}\cup S_p$ divide $C[u_p,u_{p+1}]$
into $|S_p|+1$ subsegments. By the maximality of the $xy$-cycle $C$,
every subsegment not containing the edge $xy$ has length at least $2$.
The subsegment containing $xy$ has length at least $3-\theta'-\phi$.
Hence,
\begin{equation}\label{eq:24-case1-short-seg}
|C[u_p,u_{p+1}]|\ge 3-\theta'+2|S_p|-\phi.
\end{equation}
Choose an arbitrary $q\in [t]\setminus\{p\}$. For every $k\notin\{p,q\}$,
\begin{equation}\label{eq:24-case1-three}
\dstar_H(u_k,u_{k+1})\ge 3.
\end{equation}
For every $k\neq p$, write $S\cap V(C[u_k,u_{k+1}])=\{s_1,\ldots,s_m\}$ in the order on
$C[u_k,u_{k+1}]$, and set $s_0=u_k$ and $s_{m+1}=u_{k+1}$. 
Since $C[u_k,u_{k+1}]$ avoids $xy$, the maximality of $C$
gives $|C[s_j,s_{j+1}]|\ge d_H^*(s_j,s_{j+1})$ for every $0\leq j\leq m$. For $1\le j\le m$, Lemma~\ref{lem:23}(i) gives that $s_j$ has a unique
neighbor in $H$. Hence
$d_H^*(s_0,s_{j+1})\le d_H^*(s_0,s_{j})+d_H^*(s_{j},s_{j+1})-2$ for every $1\leq j\leq m$. Applying this successively, we obtain
\[
|C[u_k,u_{k+1}]|
\ge d_H^*(u_k,u_{k+1})+2|S\cap V(C[u_k,u_{k+1}])|.
\]
Summing over $k\ne p$, and using the fact that each vertex of $S\setminus S_p$ is counted at least once, together with~\eqref{eq:24-case1-path}, \eqref{eq:24-case1-short-seg}, and~\eqref{eq:24-case1-three}, we get
\begin{align}\label{eq:24-case1-cycle-sum}
|C|
&\ge \sum_{k\notin\{p,q\}} \dstar_H(u_k,u_{k+1})
 + \dstar_H(u_q,u_{q+1}) + |C[u_p,u_{p+1}]| +2(s-|S_p|)  \\
&> 3(t-2)+\left(r-2t+2+\theta-\frac{2s-\sigma}{h}\right)+3-\theta'+2s-\phi \notag\\
&= r+t-1+\theta-\theta'+2s-\phi-\frac{2s-\sigma}{h} \ge r+1, \notag
\end{align}
a contradiction. The last inequality follows from $t\ge 2$, $\theta\ge \theta'$, and
\[
2s-\phi-\frac{2s-\sigma}{h}
=(2s-\phi)\left(1-\frac1h\right)+\frac{\sigma-\phi}{h}\ge 0.
\]

It remains to consider $h=1$, say $V(H)=\{u\}$. Write $\NC(H)=\{v_1,\dots,v_w\}$ in cyclic order on $C$, and set
$H_i:=G[V(H)\cup\{v_i,v_{i+1}\}]+v_iv_{i+1}$.
Then $H_i$ is $2$-connected, and the same calculation as in~\eqref{eq:temp-case1} gives the following:
\[
{\textstyle\sum_{v\in V(H_i)\setminus\{v_i,v_{i+1}\}}} d_{H_i}(v)
> rh-e(H,C-\{x,y,z\})-e(H,C-\{v_i,v_{i+1}\}).
\]
Set $\theta:=|\{x,y,z\}\cap\NC(H)|$. Since
$e(H,C-\{x,y,z\})\le w-\theta$ and
$e(H,C-\{v_i,v_{i+1}\})\le w-2$, we obtain
\[
\frac{{\textstyle\sum_{v\in V(H_i)\setminus\{v_i,v_{i+1}\}}} d_{H_i}(v)}{h}
>r-2w+2+\theta.
\]
Here $h=1$ and $H_i$ is a triangle, so the left-hand side equals $2$.
By~\eqref{eq:24-assume}, $r>c-1$. Moreover, all segments between
consecutive vertices of $\NC(H)$ have length at least $2$, except possibly
the one containing $xy$, which has length at least
$3-|\{x,y\}\cap\NC(H)|\ge 3-\theta$. Hence
$c\ge 2(w-1)+3-\theta=2w+1-\theta$, and therefore
\[
2>r-2w+2+\theta>(c-1)-2w+2+\theta\ge 2,
\]
a contradiction. This completes Case~\ref{Case:1}.

\begin{case}\label{Case:2}
There exists a component $H$ of $G-V(C)$ with $h=|V(H)|\ge 3$ which is not $2$-connected.    
\end{case}
Let $B$ be an end-block of $H$ with cut-vertex $b$ such that $z\notin V(B-b)$.
Set $B_0:=B-b$, $b_0:=|V(B_0)|$, and
\[
r_2:=\frac{\sum_{v\in V(B_0)} d_G(v)}{|B_0|}.
\]
Let $T:=\NC(B_0)$, let $t:=|T|$, and let $\theta:=|\{x,y,z\}\cap T|$. 
Let $G_1$ be obtained from $G$ by contracting $B$ onto its cut-vertex $b$
and deleting loops and parallel edges; equivalently, the new vertex $b$
has neighborhood $(N_G(b)\setminus N_B(b))\cup \NC(B_0)$.
Then $G_1$ is $2$-connected, $|V(G_1)|\ge 4$, and $x,y,z\in V(G_1)$.
We have
\begin{align*}
{\textstyle\sum_{v\in V(G_1)\setminus\{x,y,z\}}} d_{G_1}(v)
&\ge {\textstyle\sum_{v\in V(G)\setminus\{x,y,z\}}} d_G(v)
   -{\textstyle\sum_{v\in V(B_0)}} d_G(v)-e(B_0,b)-e(B_0,C-\{x,y,z\}) \\
&\ge (n-3)r-r_2b_0-b_0-(t-\theta)b_0 = (n-3)r-b_0(r_2+1+t-\theta).
\end{align*}
If $r\ge r_2+1+t-\theta$, then the right-hand side is at least
$(n-b_0-3)r$, so the induction hypothesis on $G_1$ gives a good
$xy$-cycle in $G_1$. If this cycle uses the contracted vertex $b$, replace
its subpath $a b a'$ by an $(a,B,a')$-path in $G$, where $a$ and $a'$ are
the two neighbors of $b$ on this cycle; this does not decrease its length.
Thus it lifts to a good $xy$-cycle in $G$, a contradiction. Hence we may assume
\begin{equation}\label{eq:24-case2-r}
 r<r_2+1+t-\theta.
\end{equation}

For any vertex $u_i\in T=\NC(B_0)$, let
$G_i:=G[V(B)\cup\{u_i\}] + u_ib.$
Then $G_i$ is $2$-connected. By~\eqref{eq:24-case2-r},
\begin{align*}
\frac{{\textstyle\sum_{v\in V(G_i)\setminus\{u_i,b\}}} d_{G_i}(v)}{|B_0|}
=\frac{{\textstyle\sum_{v\in V(B_0)}} d_G(v)-e(B_0,C-\{u_i\})}{b_0} \ge \frac{b_0r_2-b_0(t-1)}{b_0} = r_2-t+1 > r-2t+\theta.
\end{align*}
By Theorem~\ref{thm:fan-average}, $G_i$ contains a $(u_i,B_0,b)$-path in
$G_i$ of length greater than $r-2t+\theta$. 
Hence, there is a $(u_i,H,b)$-path in $G$ of length greater than $r-2t+\theta$.

\begin{claim}\label{cl:24-case2}
One of the following two alternatives holds.
\begin{enumerate}[label=(\alph*),leftmargin=2em]
\item There is a vertex $w\in \NC(H-B)$ with $T\setminus\{w\}\ne\emptyset$, and for every
$u_i\in T\setminus\{w\}$,
$\dstar_H(w,u_i)\ge r-2t+\theta+2.$
\item $T=\{u\}$ and $\NC(H-B)\subseteq\{u\}$, and there exists
$w\in N_C(b)\setminus\{u\}$ such that
$\dstar_H(u,w)\ge r-2+\theta+1.$
\end{enumerate}
\end{claim}

\begin{proof}
Since $G$ is $2$-connected and $H-B:=H-V(B)$ is non-empty, we have
$\NC(H-B)\ne\emptyset$.

If $t\ge 2$ or some vertex of $\NC(H-B)$ is not in $T$, choose $w\in \NC(H-B)$ with $T\setminus\{w\}\ne\emptyset$. Fix $u_i\in T\setminus\{w\}$, and let $P_i$ be the $(u_i,B_0,b)$-path of length greater than $r-2t+\theta$ obtained above. Choose a neighbor $w'\in V(H-B)$ of $w$ and an arbitrary $(b,H-B,w')$-path. Concatenating $P_i$, this path, and the edge $w'w$ gives an $(u_i,H,w)$-path of length at least $r-2t+\theta+2$. This proves~(a).

It remains to consider $T=\{u\}$, and therefore $\NC(H-B)\subseteq\{u\}$.
Since $G$ is $2$-connected, we have $|N_C(H)|\geq 2$.
Choose $w\in\NC(H)\setminus\{u\}$.
Since $T=\NC(B_0)=\{u\}$ and $\NC(H-B)\subseteq\{u\}$, we have $w\in N_C(b)\setminus\{u\}$.
Concatenating the $(u,B_0,b)$-path of length at least $r-2+\theta$ with $bw$ proves (b).
\end{proof}

First, suppose that Claim~\ref{cl:24-case2}(a) holds. Let
$T^+:=T\cup\{w\}$, arrange its vertices as $a_1,\ldots,a_m$ in cyclic order
on $C$, and let $\eta:=|\{x,y\}\cap T^+|$. 
Let $q$ be the index such that
$C[a_q,a_{q+1}]$ contains $xy$. 
Choose an index $p$ such that $C[a_p,a_{p+1}]$ does not contain the edge $xy$
and has endpoints $w$ and a vertex of $T\setminus\{w\}$.
The local maximality of $C$ gives $|C[a_j,a_{j+1}]|\ge 2$ for every $j\notin\{p,q\}$ and $|C[a_q,a_{q+1}]|\ge 3-\eta$.

If $w\in T$, then $m=t$, and Claim~\ref{cl:24-case2}(a) gives
$|C|\ge 2(t-2)+(r-2t+\theta+2)+3-\eta=r+1+\theta-\eta\ge r+1,$
where we use the fact $\eta=|\{x,y\}\cap T|\le \theta$. If $w\notin T$, then $m=t+1$, and similarly
$|C|\ge 2(t-1)+(r-2t+\theta+2)+3-\eta=r+3+\theta-\eta\ge r+1,$
where we use $\eta\le 2$. In both cases, this contradicts~\eqref{eq:24-assume}.

It remains to handle Claim~\ref{cl:24-case2}(b). Let $I$ be the $(u,w)$-segment of $C$ that avoids $xy$, and let $J$ be the $(w,u)$-segment. Set $\eta:=|\{x,y\}\cap\{u,w\}|$. By Claim~\ref{cl:24-case2}(b) and local maximality,
$|I|\ge r-2+\theta+1,$ and $|J|\ge 3-\eta$. Therefore
$|C|\ge r-2+\theta+1+3-\eta=r+2+\theta-\eta\ge r+1.$
For the last inequality, if $\eta=2$, then $u\in\{x,y\}$ and so $\theta\ge 1$; if $\eta\le 1$, it is immediate. This contradiction
completes Case~\ref{Case:2}.

Assume that neither Case~\ref{Case:1} nor Case~\ref{Case:2} occurs. The remaining case is the following.
\begin{case}\label{Case:3}
There is a unique component $H$ of $G-V(C)$; it contains $z$, and is either
$2$-connected, an edge, or a single vertex.    
\end{case}
Assume first that $h:=|V(H)|\ge 2$. Then $H$ is $2$-connected or an edge. Let
\[
r'_1:=\frac{\sum_{v\in V(H)\setminus\{z\}} d_G(v)}{h-1}.
\]
Let $\widehat G$ be obtained from $G$ by contracting $H$ to a single vertex $\widehat z$ and suppressing parallel edges. Then $\widehat G$ is $2$-connected. Put $\nu:=|\NC(H)\setminus\{x,y\}|$. Under the contraction, the edges between
$H$ and $C-\{x,y\}$ are suppressed to one edge for each vertex of
$\NC(H)\setminus\{x,y\}$. Hence,
\[
\sum\limits_{v\in V(\widehat G)\setminus\{x,y,\widehat z\}}d_{\widehat G}(v)
=(n-3)r-r'_1(h-1)-e(H,C-\{x,y\})+\nu.
\]

If this sum is at least $(n-h-2)r$, then the induction hypothesis applied
to $\widehat G$ gives a good $xy$-cycle in $\widehat G$. If it uses
$\widehat z$, let $a$ and $b$ be the two neighbors of $\widehat z$ on this
cycle, and replace the subpath $a\widehat z b$ by an $(a,H,b)$-path in
$G$; this does not decrease its length. Thus, every good $xy$-cycle in
$\widehat G$ lifts to a good $xy$-cycle in $G$, a contradiction. Hence
\begin{equation}\label{eq:24-case3-contract}
r'_1(h-1)>r(h-1)-e(H,C-\{x,y\})+\nu.
\end{equation}

Let $T=\{u_1,u_2,\dots,u_t\}$ be a maximum strong attachment of $H$ to $C$, in cyclic order on $C$. Let $S:=\NC(H)\setminus T$, $s:=|S|$, and $\theta:=|\{x,y\}\cap T|$. 
Since $G$ is 2-connected, we have $t\ge 2$. Choose an index $i$ such that $C[u_i,u_{i+1}]$ does not contain $xy$, and let $q$ be the index such that $C[u_q,u_{q+1}]$ contains $xy$. Set $G_i:=G[V(H)\cup\{u_i,u_{i+1}\}]+u_iu_{i+1}.$
Then $G_i$ is $2$-connected and $|V(G_i)|=h+2\ge 4$. Using~\eqref{eq:24-case3-contract},
\begin{align*}
\sum_{v\in V(G_i)\setminus\{u_i,u_{i+1},z\}} d_{G_i}(v)
&=r'_1(h-1)-e(H-\{z\},C-\{u_i,u_{i+1}\})\\
&>r(h-1)-e(H,C-\{x,y\})+\nu-e(H-\{z\},C-\{u_i,u_{i+1}\}).
\end{align*}
We claim that
\begin{equation}\label{eq:24-case3-combined-edgecount}
e(H,C-\{x,y\})-\nu+e(H-\{z\},C-\{u_i,u_{i+1}\})
\le (2t-\theta-2)(h-1)+s.
\end{equation}

To prove this, sum the contributions over vertices of $\NC(H)$. For
$v\in\NC(H)$, the contribution of $v$ to the left-hand side
of~\eqref{eq:24-case3-combined-edgecount} is
\[
\mathbf 1_{v\notin\{x,y\}}\cdot(|N_H(v)|-1)+\mathbf 1_{v\notin \{u_i,u_{i+1}\}}\cdot|N_{H-\{z\}}(v)|.
\]
If $v\in S$, then $|N_H(v)|=1$ by Lemma~\ref{lem:23}(i), and this contribution
is at most $1$. If $v\in T\setminus\{u_i,u_{i+1},x,y\}$, it is at most
$2(h-1)$; if $v\in (\{u_i,u_{i+1}\}\triangle\{x,y\})\cap T$, it is at most $h-1$; and
if $v\in \{u_i,u_{i+1}\}\cap\{x,y\}$, it is $0$. 
Since
$2|T\setminus(\{u_i,u_{i+1}\}\cup\{x,y\})|+|(\{u_i,u_{i+1}\}\triangle\{x,y\})\cap T|=2t-\theta-2$,
summing these contributions gives~\eqref{eq:24-case3-combined-edgecount}.

It follows that
\[
\frac{{\textstyle\sum\limits_{v\in V(G_i)\setminus\{u_i,u_{i+1},z\}}} d_{G_i}(v)}{h-1}
>r-2t+\theta+2-\frac{s}{h-1}.
\]
By the induction hypothesis on $G_i$, there is a $(u_i,u_{i+1})$-path in
$G_i$ of length greater than the right-hand side, and hence a
$(u_i,H,u_{i+1})$-path in $G$ of length greater than the same quantity.
The maximality of $C$ gives
\begin{equation*}\label{eq:24-case3-local-path}
\dstar_H(u_i,u_{i+1})>r-2t+\theta+2-\frac{s}{h-1}.
\end{equation*}

Put $S_q:=S\cap V(C[u_q,u_{q+1}])$ and
$\phi:=|\{x,y\}\cap S_q|$. The same argument as for~\eqref{eq:24-case1-short-seg} gives
\[
|C[u_q,u_{q+1}]|\ge 3-\theta+2|S_q|-\phi.
\]
Using the same segment summation as in~\eqref{eq:24-case1-cycle-sum}, we obtain
\begin{align*}
|C|
&\ge \sum_{k\notin\{i,q\}}\dstar_H(u_k,u_{k+1})
   +\dstar_H(u_i,u_{i+1})+|C[u_q,u_{q+1}]|+2(s-|S_q|)\\
&>3(t-2)+\left(r-2t+\theta+2-\frac{s}{h-1}\right)
  +3-\theta+2s-\phi\\
&=(r+1)+(t-2)+(2s-\frac{s}{h-1}-\phi).
\end{align*}
The last expression is at least $r+1$, since
$t\ge 2$, and
$2s-\frac{s}{h-1}\ge s\geq \phi$ for all $h\ge 2$.
This contradicts~\eqref{eq:24-assume} and completes the subcase $h\ge 2$.

Finally, let $h=1$. Then $V(H)=\{z\}$ and $c=n-1$. Write
$U:=\NC(z)=\{u_1,u_2,\dots,u_d\}$ in cyclic order on $C$. Since $G$ is
$2$-connected, $d\ge 2$. It remains to prove $r\le c-1=n-2$, which is
equivalent to
\begin{equation}\label{eq:24-h1-target}
\sum_{v\in V(G)\setminus\{x,y,z\}}d(v)\le (n-3)(n-2).
\end{equation}

Assume first that either $d\ge 3$, or $d=2$ and $U\cap\{x,y\}\ne\emptyset$. Choose an orientation of $C$ such that, when
$\{x,y\}\nsubseteq U$, no edge $u^-u$ with $u\in U$ equals $xy$ (if
$\{x,y\}\subseteq U$, the orientation is arbitrary), where $v^-$ and
$v^+$ denote the predecessor and successor of $v$ in this orientation.
Let $P:=\{u\in V(C):u^+\in U,\ uu^+\ne xy\}$ and set
$m:=|P|=d-\mathbf 1_{\{x,y\}\subseteq U}$.
By definition, every $p\in P$ has $p^+\in U$, and the edge $pz$ cannot belong to $G$, as otherwise replacing the edge $pp^+$ of $C$ by the path $pzp^+$ gives an $xy$-cycle of length $n$.
The same replacement argument also gives $U\cap P=\emptyset$.
Moreover, $P$ is independent: if $p_1p_2\in E(G)$ for $p_1,p_2\in P$, then replacing the two edges $p_1p_1^+$ and $p_2p_2^+$
of $C$ by $p_1p_2$ and the path $p_1^+zp_2^+$ gives an $xy$-cycle of length $n$.
Hence every $p\in P$ has degree at most $n-m-1$ in $G$.

Let $a:=|U\cap\{x,y\}|$ and $q_0:=|P\cap\{x,y\}|$.
By assumption, if $d=2$, then $a\geq 1$.
We estimate the degree-sum in $V(G)\setminus\{x,y,z\}$ by separating the vertices in $U\setminus\{x,y\}$, $P\setminus\{x,y\}$, and the remaining vertices. These three classes have sizes $d-a$,
$m-q_0$, and $n-3-(d-a)-(m-q_0)$, respectively, and their degrees are at
most $n-1$, $n-m-1$, and $n-2$, respectively.
Hence
\begin{align*}
\sum_{v\in V(G)\setminus\{x,y,z\}}d(v)
&\le (d-a)(n-1)+(m-q_0)(n-m-1)+\bigl(n-3-(d-a)-(m-q_0)\bigr)(n-2)\\
&=(n-3)(n-2)+(d-a)-(m-q_0)(m-1).
\end{align*}
If $m=d$, then our orientation gives $q_0\le 1$, so $(m-q_0)(m-1)\ge (d-1)^2\ge d-a$ (recall $a\geq 1$ when $d=2$). If $m=d-1$, then $a=2$ and $q_0=0$, so $(m-q_0)(m-1)=(d-1)(d-2)\ge d-2=d-a$. Thus, \eqref{eq:24-h1-target} holds for $d\ge 3$, and for
$d=2$ with $U\cap\{x,y\}\ne\emptyset$.

It remains to assume $d=2$ and $U\cap\{x,y\}=\emptyset$, say
$U=\{u_1,u_2\}$. Let $Q$ be the $(u_1,u_2)$-segment of $C$ containing
the edge $xy$, and let $R$ be the other $(u_1,u_2)$-segment. Then
$|R|\ge 2$. Let $p_R$ and $q_R$ be the neighbors of $u_1$ and $u_2$ on
$R$, respectively, and let $s_Q$ and $t_Q$ be the neighbors of $u_1$ and
$u_2$ on $Q$, respectively.

Neither $p_Rt_Q$ nor $q_Rs_Q$ is an edge of $G$. Indeed, if
$p_Rt_Q\in E(G)$, then replacing the edges $u_1p_R$ and $u_2t_Q$ of $C$
by the paths $u_1zu_2$ and $p_Rt_Q$ gives an $xy$-cycle of length $n$;
the proof for $q_Rs_Q$ is symmetric. Since $U\cap\{x,y\}=\emptyset$, the
vertices $u_1,u_2,p_R,q_R$ all belong to $V(G)\setminus\{x,y,z\}$.
Among the $n-3$ vertices counted in~\eqref{eq:24-h1-target}, only
$u_1$ and $u_2$ may be adjacent to $z$. Moreover, the non-edges $p_Rt_Q$ and $q_Rs_Q$ reduce the total
contribution of $p_R$ and $q_R$ by 2; when $p_R=q_R$, this is still
true because the edge $xy$ lies on $Q$, so $s_Q\ne t_Q$.
Thus
\[
\sum_{v\in V(G)\setminus\{x,y,z\}} d_G(v)
\le 2(n-1)+(n-5)(n-2)-2=(n-3)(n-2).
\]
This proves~\eqref{eq:24-h1-target}, and hence completes the proof of
Theorem~\ref{thm:24}.
\end{proof}

The same attachment argument, with Theorem~\ref{thm:24} in place of the bound in Theorem~\ref{thm:fan-average}, gives the following useful variant of Lemma~\ref{lem:23}(iii).

\begin{theorem}\label{thm:25}
Let $G$ be a graph, let $C$ be a cycle of $G$, and let $H$ be a component of $G-V(C)$. Assume that $H$ is $2$-connected and that $C$ is locally longest with respect to $H$. Let $T=\{u_1,\dots,u_t\}$ be a maximum strong attachment of $H$ to $C$, let $s:=|\NC(H)\setminus T|$, let $v\in V(H)$, and set $h:=|V(H)|$ and
\[
r:=\frac{\sum_{x\in V(H)\setminus\{v\}} d_G(x)}{h-1}.
\]
Then for every strongly attached pair $\{u_i,u_{i+1}\}$,
$\dstar_H(u_i,u_{i+1})\ge r+2-t-\frac{s}{h-1}.$
\end{theorem}

\begin{proof}
Fix $i$ and let
$G_i:=G[V(H)\cup\{u_i,u_{i+1}\}] + u_iu_{i+1}.$
Then $G_i$ is $2$-connected and has $h+2$ vertices. Apply Theorem~\ref{thm:24} to $G_i$ with distinguished vertices $u_i,u_{i+1},v$. The relevant average-degree term equals
\begin{align*}
\frac{\sum\limits_{x\in V(G_i)\setminus\{u_i,u_{i+1},v\}} d_{G_i}(x)}{h-1}
&=\frac{\sum\limits_{x\in V(H)\setminus\{v\}} d_G(x)-e(V(H)\setminus\{v\},C-\{u_i,u_{i+1}\})}{h-1} \\
&\ge r-\frac{(h-1)(t-2)+s}{h-1}= r+2-t-\frac{s}{h-1},
\end{align*}
where the inequality follows from Lemma~\ref{lem:23}(i). 
By Theorem~\ref{thm:24}, this gives a $(u_i,u_{i+1})$-path in $G_i$,
and thus a $(u_i,H,u_{i+1})$-path in $G$ of length at least
$r+2-t-\frac{s}{h-1}$, proving the desired bound.
\end{proof}

\section{A lemma on DFS-trees}\label{Sec:3}

A {\it DFS-tree} is a spanning tree of a given graph obtained by applying depth-first search to the graph. 
Let $H$ be a connected graph and let $T$ be a DFS-tree of $H$ rooted at $r$. For each
vertex $u$, let $T_u:=T[u,r]$ be the unique $(u,r)$-path in $T$. For $x,y\in V(T)$, if $y\in V(T_x)$, we write $y\le_T x$; if $x\ne y$, we
also write $y<_T x$. Two vertices $x,y\in V(T)$ are called \emph{comparable} if either
$x\le_T y$ or $y\le_T x$, and \emph{incomparable} otherwise. We shall use
the following standard property of DFS-trees.

\begin{proposition}\label{prop:31}
Let $H$ be a connected graph and let $T$ be a DFS-tree of $H$. For every edge $xy\in E(H)$, the vertices $x$ and $y$ are comparable in $T$.
\end{proposition}

We use the following convention for choosing a DFS-tree. For a path
$P=v_1v_2\cdots v_p$ in a connected graph $H$, we root the DFS-tree at
$v_1$ and, whenever $v_i$ is explored for some $i<p$, visit $v_{i+1}$
before all other unvisited neighbors. Then $P$ is an initial root path
of the resulting DFS-tree. In particular, $v_1,\ldots,v_p$ are the first
$p$ vertices in the DFS order, and every vertex of $P$ with degree at least three in $T$ has a child outside $V(P)$.

The main result of this section is Lemma~\ref{lem:31}.  It gives a set $U\subseteq V(H)$ of size $p(H)$ whose degree sum is small. 
We first describe the construction of $U$.

\begin{definition}\label{def:dfs-construction}
Let $P=v_1v_2\cdots v_p$ be a longest path in $H$. Choose a DFS-tree $T$ rooted at $r=v_1$ such that the vertices $v_1,v_2,\dots,v_p$ are the first $p$ vertices in the DFS order. 
We define an injection\footnote{The map is indeed injective. In (R1), the chosen vertices $u_i$ belong to vertex-disjoint paths $T[v_i,u_i]$. In (R2), the incomparability of $U_1\cup W$ implies that the leaves $u_i$ are pairwise distinct and disjoint from $U_1$. Finally, $U_1\cup U_2\subseteq V(H-P)$, while $U_3\subseteq V(P)$.} $f:V(P)\to V(H)$ as follows.

\begin{enumerate}[label=(R\arabic*),leftmargin=2.5em]
\item Let
\(I_1:=\{i\in [p]: \text{the vertex } v_i \text{ has degree at least three in the tree } T\}.
\)
For each $i\in I_1$, choose an arbitrary child $x_i$ of $v_i$ in $T$ with
$x_i\notin V(P)$, and choose an arbitrary leaf $u_i$ of $T$ such that
$x_i\le_T u_i$. Define $f(v_i)=u_i$, and let $U_1:=\{u_i:i\in I_1\}.$

\item Choose a set $I_2\subseteq [p]\setminus I_1$ and pairwise distinct
vertices $w_i$ for $i\in I_2$, and set $W=\{w_i:i\in I_2\}$, with
$w_i\in N_H(v_i)\setminus (V(P)\cup U_1)$
for each $i\in I_2,$
under the following conditions:
\begin{enumerate}[label=(\roman*),leftmargin=2em]
\item any two distinct vertices in $U_1\cup W$ are incomparable in $T$;
\item subject to (i), $|I_2|$ is maximum;
\item subject to (i) and (ii), the rearranged non-decreasing sequence of the numbers
$|T[r,w_i]|$, where $i\in I_2$, is lexicographically maximum.
\end{enumerate}

After fixing $I_2$ and $W$, for each $i\in I_2$, choose an arbitrary leaf
$u_i\in V(H-P)$ of $T$ with $w_i\le_T u_i$, so that
$\sum_{i\in I_2}|T[w_i,u_i]|$ is maximized. Let
$U_2=\{u_i:i\in I_2\}$. For each $i\in I_2$, define $f(v_i)=u_i$.

\item For each $i\in I_3:=[p]\setminus (I_1\cup I_2)$, define $f(v_i)=v_i$, and let $U_3:=\{v_i:i\in I_3\}.$
\end{enumerate}

Let $U:=\{f(v_i): i\in [p]\}$, and for convenience write $u_i:=f(v_i)$ for each $i\in [p]$.
\end{definition}

The next notation defines the paths by which the vertices in $U_1\cup U_2$
return to $P$.

\begin{definition}\label{def:dfs-paths}
For each $j\in I_1\cup I_2$, let $\ell_j\in [p]$ be the maximum index such that
$v_{\ell_j}\in V(T_{u_j})$. Set
$Q_j:=T[v_{\ell_j},u_j],$ $\widehat Q_j:=Q_j-\{v_{\ell_j}\}.$
Define a collection of vertex-disjoint\footnote{To see the disjointness: For $j\in I_1$, we have $\ell_j=j$; if $x_j$ is the child chosen in (R1), then
$Q'_j=T[x_j,u_j]$, so these paths are pairwise vertex-disjoint. For distinct
$i,j\in I_2$, the vertices $w_i$ and $w_j$ are incomparable by (R2), so
$T[w_i,u_i]$ and $T[w_j,u_j]$ are disjoint. Finally, if
$a\in I_1$ and $b\in I_2$, then the incomparability of $w_b$ and
$u_a$ implies that $T[x_a,u_a]$ and $T[w_b,u_b]$ are disjoint.} paths by
\[
Q'_j:=
\begin{cases}
\widehat Q_j, & j\in I_1,\\
T[w_j,u_j], & j\in I_2.
\end{cases}
\]
\end{definition}

\begin{proposition}\label{prop:32}
There exist pairwise vertex-disjoint $(u_i,v_i)$-paths in $H$ for
all $i\in[p]$, each of which has no internal vertex in $V(P)$.
\end{proposition}

\begin{proof}
We define the desired $(u_i,v_i)$-path $R_i$ by: 
\[
R_i:=
\begin{cases}
Q_i, & i\in I_1,\\
Q'_i\cup\{w_iv_i\}, & i\in I_2,\\
\text{the trivial path at } v_i=u_i, & i\in I_3.
\end{cases}
\]

Each path $R_i$ intersects $P$ only at its endpoint $v_i$.
Since the paths $Q'_i$ are pairwise vertex-disjoint and avoid $V(P)$, and the added vertices $v_i$ are distinct, we see that the paths $R_i (i\in I_1\cup I_2)$ are pairwise vertex-disjoint. 
The paths with $i\in I_3$ are the trivial paths at the remaining
vertices of $P$. 
Hence $R_1,\dots,R_p$ are pairwise vertex-disjoint, and each has no internal vertex in $V(P)$.
\end{proof}

We now locate the neighbors of the vertices indexed by $I_3$, for the
degree estimate below.

\begin{proposition}\label{prop:33}
For each $i\in I_3$,
$N_H(v_i)\subseteq V(P)\cup \bigcup_{j\in I_1\cup I_2} V(\widehat Q_j).$
\end{proposition}

\begin{proof}
Suppose for a contradiction that there exist $i\in I_3$ and
$x\in N_H(v_i)$ such that
$x\notin V(P)\cup \bigcup_{\ell\in I_1\cup I_2} V(\widehat Q_\ell).$
By Proposition~\ref{prop:31}, the vertices $v_i$ and $x$ are comparable
in $T$. Since $x\notin V(P)$ and $P$ is the initial root path of $T$, we
have $v_i<_T x$.

By the maximality of $|I_2|$, the vertex $x$ must be comparable with some
vertex of $U_1\cup W$; otherwise we could add $i$ to $I_2$ and set
$w_i:=x$, preserving (R2)(i) while increasing $|I_2|$.
For every $\ell\in I_1\cup I_2$, every ancestor in $T$ of a vertex of
$\widehat Q_\ell$ lies in $V(P)\cup V(\widehat Q_\ell)$. In particular,
since $U_1\cup W\subseteq \bigcup_{\ell\in I_1\cup I_2}V(\widehat Q_\ell)$,
the vertex $x$ is not an ancestor of any vertex of $U_1\cup W$. Since the vertices of
$U_1$ are leaves of $T$, the comparable vertex must be some $w_j\in W$,
and necessarily $w_j<_T x$.

Now form $I_2'=(I_2\setminus\{j\})\cup\{i\}$, keep all $w_h$ with
$h\in I_2\setminus\{j\}$ unchanged, and set $w_i:=x$. The new vertex
$x$ is incomparable with every vertex of $(U_1\cup W)\setminus\{w_j\}$;
otherwise that vertex would also be comparable with the ancestor $w_j$,
contrary to (R2)(i). Thus (R2)(i) is preserved, and $|I_2'|=|I_2|$.
Since $w_j<_T x$, we have $|T[r,x]|>|T[r,w_j]|$, so the non-decreasing
sequence in (R2)(iii) is lexicographically larger, contradicting the
choice of $W$.
\end{proof}

We now derive the degree estimates.
\begin{lemma}\label{lem:31}
Let $P=v_1v_2\cdots v_p$ be a longest path in a connected graph $H$, and construct $U=\{u_i:i\in [p]\}$ from $P$ as in Definition~\ref{def:dfs-construction}. Then the following holds.
\begin{enumerate}[label=(\roman*),leftmargin=2.2em]
\item If $P$ is not a Hamiltonian path of $H$, then
$\sum_{i\in [p]} d_H(u_i)\le p(p-2).$
\item In all cases,
$\sum_{i\in [p]} d_H(u_i)\le (2c(H)-2)(p-1).$
Here, when $H$ is a tree, we use the convention $c(H)=2$.
\end{enumerate}
\end{lemma}

\begin{proof}
We retain the notation $I_1,I_2,I_3,Q_j,Q'_j,\widehat Q_j,\ell_j$
introduced above. By Proposition~\ref{prop:33}, for each $i\in I_3$
\[
d_H(v_i)\le d_P(v_i)+
\sum_{j\in I_1\cup I_2} d_{\widehat Q_j}(v_i).
\]
Summing this inequality over $i\in I_3$, and then adding
$\sum_{j\in I_1\cup I_2}d_H(u_j)$ to both sides, gives
\[
\sum_{i\in [p]}d_H(u_i)
\le
\sum_{i\in I_3} d_P(v_i)
+\sum_{j\in I_1\cup I_2}
\left(d_H(u_j)+\sum_{i\in I_3} d_{\widehat Q_j}(v_i)\right).
\]

The proof reduces to verifying the following three estimates:
\begin{itemize}
    \item[(a)] If $P$ is not
Hamiltonian, then $\sum_{i\in I_3} d_P(v_i)
\le |I_3|(p-2).$
    \item[(b)] $\sum_{i\in I_3} d_P(v_i)
\le (2|I_3|-2)(c(H)-1).$
    \item[(c)] For each $j\in I_1\cup I_2$, $d_H(u_j)+\sum_{i\in I_3} d_{\widehat Q_j}(v_i)
\le \min\{p-2,\,2c(H)-2\}.$
\end{itemize}

We will use the fact that $1,p\in I_3$ throughout these estimates. 
This is because $P$ is longest, the endpoints $v_1$ and $v_p$ have no neighbors in
$H-V(P)$, and hence they belong neither to $I_1$ nor to $I_2$.

\medskip

\noindent\textbf{Proof of~(a).}
Assume that $P$ is not Hamiltonian. Then $H-V(P)\ne\emptyset$ and $p\ge 3$.
By the standard P\'osa rotation argument, no predecessor on $P$ of a
neighbor of $v_1$ is adjacent to $v_p$; otherwise the vertices of $P$
would form a cycle, which could be extended to a path longer than $P$. Therefore,
\begin{equation}\label{eq:31-posa}
d_P(v_1)+d_P(v_p)\le p-1.
\end{equation}

If $|I_3|=2$, then $I_3=\{1,p\}$ and~\eqref{eq:31-posa} gives
$\sum_{i\in I_3} d_P(v_i)\le p-1\le 2(p-2).$
If $|I_3|=3$, then $I_1$ is non-empty as $P$ is not Hamiltonian.
This implies $p\geq 4$, and
$\sum_{i\in I_3} d_P(v_i)\le (p-1)+(p-1)\le 3(p-2).$
Now suppose $|I_3|\ge 4$. For every $i\in I_3\setminus\{1,p\}$,
$d_P(v_i)\le (p-3)+\mathbf{1}_{v_iv_1\in E(H)}+\mathbf{1}_{v_iv_p\in E(H)}.$

Summing over $i\in I_3\setminus\{1,p\}$ and using~\eqref{eq:31-posa}, we obtain the desired bound:
\begin{align*}
{\textstyle\sum_{i\in I_3}} d_P(v_i)
&\le d_P(v_1)+d_P(v_p)+(|I_3|-2)(p-3)
   +|N_P(v_1)\cap U_3|+|N_P(v_p)\cap U_3| \\
&\le (p-1)+(|I_3|-2)(p-3)+(p-1) = |I_3|(p-3)+4\le |I_3|(p-2).\qedhere
\end{align*}

\noindent\textbf{Proof of~(b).}
We first show that
$d_P(v_1)\le c(H)-1, d_P(v_p)\le c(H)-1.$
Indeed, if $d_P(v_1)\le 1$ then this follows from $c(H)\ge 2$; otherwise, the farthest neighbor of $v_1$ on $P$, together with the corresponding segment of $P$, gives a cycle of length at least $d_P(v_1)+1$. The proof for $v_p$ is symmetric.

For every $i\in I_3\setminus\{1,p\}$, let $a_i$ and $b_i$ be the numbers
of neighbors of $v_i$ on $P[v_1,v_{i-1}]$ and on $P[v_{i+1},v_p]$,
respectively. Each of $a_i$ and $b_i$ is at most $c(H)-1$: if
$a_i\le 1$, this follows from $c(H)\ge 2$; and if $a_i\ge 2$, then the
two extreme neighbors of $v_i$ on $P[v_1,v_{i-1}]$, together with the
segment of $P$ between them and the two edges from them to $v_i$, form a
cycle of length at least $a_i+1$. The proof for $b_i$ is symmetric.
Hence
$d_P(v_i)=a_i+b_i\le 2(c(H)-1).$
Therefore, we derive the desired bound:
\begin{equation*}
{\textstyle\sum_{i\in I_3}} d_P(v_i)=d_P(v_1)+d_P(v_p)+{\textstyle\sum_{i\in I_3\setminus\{1,p\}}} d_P(v_i)\le (2|I_3|-2)(c(H)-1). \tag*{\qedsymbol}
\end{equation*}

\noindent\textbf{Proof of~(c).}
Fix $j\in I_1\cup I_2$. We first introduce some notation. Let 
$P_j:=P[v_1,v_{\ell_j}],$
$Y_j:=\{y\in V(P_j)\setminus\{v_{\ell_j}\}: N_{\widehat Q_j}(y)\ne\emptyset\}.$
List the vertices of $Y_j$ as $y_1,\dots,y_t$ in their order on $P$ from $v_1$ to $v_p$.

We first claim that every vertex $v_i$ with $i\in I_3$ and
$N_{\widehat Q_j}(v_i)\ne\emptyset$ lies in
$V(P_j)\setminus\{v_{\ell_j}\}$, and hence belongs to $Y_j$.
By Proposition~\ref{prop:31}, the vertex $v_i$ is comparable with a
vertex of $\widehat Q_j$; since $v_i\in V(P)$, this gives
$v_i\in V(P_j)$. It remains to show that $v_i\ne v_{\ell_j}$. Indeed,
$\ell_j\notin I_3$: this is clear if $j\in I_1$, since then
$\ell_j=j\notin I_3$; if $j\in I_2$, then the maximality of $P$ gives
$v_{\ell_j}\notin\{v_1,v_p\}$, and $Q_j$ witnesses that
$\ell_j\in I_1$. This proves the claim, which further implies the following: 
\begin{equation}\label{eq:trans}
\sum_{i\in I_3}d_{\widehat Q_j}(v_i)\le \sum_{\alpha=1}^t d_{\widehat Q_j}(y_\alpha).
\end{equation}

We now prove~(c) for this fixed $j$. Since $u_j$
is a leaf of $T$, Proposition~\ref{prop:31} implies that
$N_H(u_j)\subseteq V(T[v_1,u_j])$. If $d_H(u_j)\ge 2$, then the first
and last neighbors of $u_j$ on this path, together with the segment
between them and the two edges from them to $u_j$, form a cycle of length
at least $d_H(u_j)+1$; the case $d_H(u_j)\le 1$ follows from
$c(H)\ge 2$. Hence,
\begin{equation}\label{eq:31-leaf-circ}
d_H(u_j)\le c(H)-1.
\end{equation}

If $Y_j=\emptyset$, then the preceding claim gives
$\sum_{i\in I_3}d_{\widehat Q_j}(v_i)=0.$
Moreover, $v_1\notin N_H(u_j)$, as otherwise $u_jv_1P[v_1,v_p]$ would be longer
than $P$. This gives the desired bound with respect to $p$:
$d_H(u_j)\le |T[v_1,u_j]|-1\le p-2.$
The circumference bound follows from~\eqref{eq:31-leaf-circ}.
This proves~(c) when $Y_j=\emptyset$.

We may therefore assume that $Y_j\neq \emptyset$, so $t\ge 1$. For $1\leq s\leq t$, define
$a_s:=d_{\widehat Q_j}(y_s).$
For $1\le s<t$, apply Lemma~\ref{lem:21} to $\widehat Q_j$ with the two
sets $N_{\widehat Q_j}(y_s)$ and $N_{\widehat Q_j}(y_{s+1})$, choosing
the two neighbors so that their distance on $\widehat Q_j$ is maximum.
Together with the two edges from $y_s$ and $y_{s+1}$ to these chosen
neighbors, this gives a $(y_s,\widehat Q_j,y_{s+1})$-path of length at least $(a_s+a_{s+1})/2+1$. The maximality of $P$ then implies
\begin{equation}\label{eq:31-middle}
|P[y_s,y_{s+1}]|\ge
\frac{a_s+a_{s+1}}{2}+1.
\end{equation}

For the last segment, take a neighbor of $y_t$ on $\widehat Q_j$ whose
distance from $v_{\ell_j}$ along $Q_j$ is maximum. Since this distance is
at least $a_t$, the resulting $(y_t,Q_j,v_{\ell_j})$-path  has length at least $a_t+1$. The maximality of $P$ yields
\begin{equation}\label{eq:31-last}
|P[y_t,v_{\ell_j}]|\ge a_t+1.
\end{equation}
Moreover, $N_{P_j}(u_j)\subseteq Y_j\cup\{v_{\ell_j}\}$, so
\begin{equation}\label{eq:31-dPj}
d_{P_j}(u_j)\le t+1.
\end{equation}
Combining~\eqref{eq:31-middle}, \eqref{eq:31-last}, and
\eqref{eq:31-dPj}, and using $a_1,a_t\ge 1$, we obtain the following:
\begin{align}
a_1+|P[y_1,v_{\ell_j}]|
&=a_1+\sum_{s=1}^{t-1}|P[y_s,y_{s+1}]|+|P[y_t,v_{\ell_j}]|
\ge a_1+\sum_{s=1}^{t-1}
\left(\frac{a_s+a_{s+1}}{2}+1\right)+(a_t+1) \notag\\
&= \sum_{s=1}^t a_s+\frac{a_1+a_t}{2}+t
\ge \sum_{s=1}^t a_s+d_{P_j}(u_j).
\label{eq:31-key}
\end{align}

We now finish the two bounds in~(c) separately, using two different choices from
$N_{\widehat Q_j}(y_1)$.
For the bound $p-2$, consider the first segment and take a neighbor of $y_1$ on $\widehat Q_j$ whose
distance from $u_j$ along $Q_j$ is maximum. This gives a
$(u_j,Q_j,y_1)$-path of length at least $a_1$. By the maximality of $P$,
\begin{equation}\label{eq:31-first}
|P[v_1,y_1]|\ge a_1.
\end{equation}
Adding~\eqref{eq:31-first} to~\eqref{eq:31-key} gives
$|P_j|=|P[v_1,v_{\ell_j}]|
=|P[v_1,y_1]|+|P[y_1,v_{\ell_j}]|
\ge \sum_{s=1}^t a_s+d_{P_j}(u_j).$
By Proposition~\ref{prop:31}, we have $N_{H-P_j}(u_j)\subseteq V(Q_j-\{u_j,v_{\ell_j}\})$, and therefore
$d_{H-P_j}(u_j)\le |Q_j|-1.$
Combined with~\eqref{eq:trans} (recall $a_s=d_{\widehat Q_j}(y_s)$), it follows that
\begin{align*}
d_H(u_j)+\sum_{i\in I_3} d_{\widehat Q_j}(v_i)&\leq 
d_H(u_j)+\sum_{s=1}^t a_s=d_{H-P_j}(u_j)+\big(d_{P_j}(u_j)+\sum_{s=1}^t a_s\big)\\
&\le |P_j|+(|Q_j|-1)=|T[v_1,u_j]|-1
\le p-2.
\end{align*}

It remains to prove the upper bound $2c(H)-2$ in~(c).
Choose a neighbor of $y_1$ on $\widehat Q_j$ farthest from $v_{\ell_j}$
along $Q_j$. The edge from $y_1$ to this neighbor, the segment of $Q_j$
from this neighbor to $v_{\ell_j}$, and the segment $P[y_1,v_{\ell_j}]$
form a cycle of length at least $a_1+|P[y_1,v_{\ell_j}]|+1$. Hence,
\begin{equation*}\label{eq:31-circ1}
a_1+|P[y_1,v_{\ell_j}]|\le c(H)-1.
\end{equation*}
Together with~\eqref{eq:31-key}, this gives
$\sum_{s=1}^t a_s\le c(H)-1.$
Combining this with~\eqref{eq:31-leaf-circ}, we obtain
$d_H(u_j)+\sum_{s=1}^t a_s\le 2c(H)-2.$
Together with~\eqref{eq:trans}, this proves the bound $2c(H)-2$, and
therefore completes the proof of~(c).\qed

\medskip

Now we return to the proof of Lemma~\ref{lem:31}. 
If $P$ is not Hamiltonian, using~(a) and the
$p-2$ bound in~(c) over
$j\in I_1\cup I_2$ gives
\[
\sum_{i\in[p]}d_H(u_i)
\le \bigl(|I_3|+|I_1|+|I_2|\bigr)(p-2)
=p(p-2).
\]
This proves~(i). In all cases, using~(b)
and the $2c(H)-2$ bound in~(c) over
$j\in I_1\cup I_2$ gives
\[
\sum_{i\in[p]}d_H(u_i)
\le \bigl(|I_3|-1+|I_1|+|I_2|\bigr)(2c(H)-2)
=(p-1)(2c(H)-2).
\]
This proves (ii) and Lemma~\ref{lem:31}. \qed
\end{proof}

\section{Proofs of Theorems~\ref{thm:local-bondy} and~\ref{thm:local-jung}}\label{Sec:4}
In this section, we first complete the proof of Theorem~\ref{thm:local-jung} and then derive Theorem~\ref{thm:local-bondy}. We recall that Table~\ref{tab:roadmap} summarizes the main steps in the proof of Theorem~\ref{thm:local-jung}.

\subsection{Small $p(H)$}\label{sec:small p}
\begin{theorem}\label{thm:41}
Let $(G,C,H)$ be a $k$-triple system and
$\delta:=\delta_G(H)=\min\{d_G(v): v\in V(H)\}$.
If $k-1\le p(H)\le \delta-k+1$, then
$|C|\ge k(\delta-k+2)$.
\end{theorem}

\begin{proof}
Let $p:=p(H)$, $P=v_1v_2\cdots v_p$ be a longest path in $H$, and set
$f(x):=x(\delta+2-x)$. Since $f$ is concave and symmetric about $(\delta+2)/2$, and since
$k\leq p+1\leq\delta-k+2$, we have
$f(p+1)\ge f(k)=f(\delta-k+2)=k(\delta-k+2)$.

We first establish the following claim. Suppose that
$U=\{u_1,\dots,u_p\}\subseteq V(H)$ and that there are pairwise
vertex-disjoint $(u_i,v_i)$-paths $R_i$ in $H$ with
$V(R_i)\cap V(P)=\{v_i\}$ for each $i$. If $|\NC(U)|
\geq 2$, then
\begin{equation}\label{eq:41-counting}
|C|\ge e(U,C)+|\NC(U)|.
\end{equation}
Indeed, write $N_C(U)=\{x_1,x_2,\dots,x_t\}$ in cyclic order on $C$, with $x_{t+1}=x_1$. For each $i\in [t]$, define
$A_i:=\{v_j: u_j\in N_U(x_i)\},$
$B_i:=\{v_j: u_j\in N_U(x_{i+1})\}.$
Let $d_i:=\max\{|P[v_a,v_b]|:v_a\in A_i,\ v_b\in B_i\}$. By Lemma~\ref{lem:21}, $d_i\ge (|A_i|+|B_i|)/2-1$. If $d_i=0$, then $A_i=B_i=\{v_a\}$ for some $a$, and the path $x_i u_a x_{i+1}$ has length $2=(d_U(x_i)+d_U(x_{i+1}))/2+1$. If $d_i\ge 1$, choose $v_a\in A_i$ and $v_b\in B_i$ with $|P[v_a,v_b]|=d_i$. Since $a\ne b$, concatenating the edge $x_iu_a$, the path $R_a$, the segment $P[v_a,v_b]$, the path $R_b$ in reverse, and the edge $u_bx_{i+1}$ produces a simple $(x_i,H,x_{i+1})$-path of length at least
$1+d_i+1
\ge \frac{d_U(x_i)+d_U(x_{i+1})}{2}+1.$
Since $C$ is locally longest with respect to $H$, we obtain
$|C[x_i,x_{i+1}]|\ge \frac{d_U(x_i)+d_U(x_{i+1})}{2}+1.$
Summing over $i=1,\dots,t$ gives
$|C|=\sum_{i=1}^t |C[x_i,x_{i+1}]|
\ge \sum_{i=1}^t d_U(x_i)+t
= e(U,C)+|N_C(U)|.$ This proves~\eqref{eq:41-counting}.

We then distinguish two cases according
to whether $P$ is a Hamiltonian path of $H$. 
First, assume that $P$ is not Hamiltonian.
By Proposition~\ref{prop:32} and
Lemma~\ref{lem:31}(i), there is a set
$U=\{u_1,\dots,u_p\}\subseteq V(H)$ satisfying the hypothesis of the claim and
$\sum_{i=1}^p d_H(u_i)\le p(p-2)$. Hence
$e(U,C)=\sum_{i=1}^p(d_G(u_i)-d_H(u_i))\ge p\delta-p(p-2)
=p(\delta-p+2)$. 
Since each vertex of $\NC(U)$ contributes at most $p$ to $e(U,C)$, we have $|\NC(U)|\ge e(U,C)/p\geq\delta-p+2$.
In particular, $|\NC(U)|\geq \delta-p+2\geq k+1>2$.
Then~\eqref{eq:41-counting} gives $|C|\ge e(U,C)+|\NC(U)|\geq(p+1)(\delta-p+2)>f(p+1)\geq f(k)$, which suffices.

It remains to assume that $P$ is Hamiltonian.
Then $|V(H)|=p$, and every
vertex of $H$ has at least $\delta-(p-1)$ neighbors on $C$.
Thus
$e(H,C)\ge p(\delta-p+1)$ and $|\NC(H)|\ge \delta-p+1\ge 2$.
Applying~\eqref{eq:41-counting} with $U=V(H)$ gives
$|C|\ge e(H,C)+|\NC(H)|\ge (p+1)(\delta-p+1)=f(p+1)\ge f(k)$.
This proves the theorem.
\end{proof}

\subsection{Large circumference}\label{sec:large c}
A \emph{block} of a connected graph is a maximal connected subgraph with no cut-vertex. An \emph{end-block} is a block containing at most one cut-vertex of the graph. If $D$ is an end-block, we write $x\in V(D)$ for its unique cut-vertex, or for an arbitrary vertex when $D$ has no cut-vertex. The \emph{block structure} $T(H)$ of a connected graph $H$ is the bipartite graph with bipartition $(\mathscr{B}(H),\mathscr{C}(H))$, where $\mathscr{B}(H)$ is the set of blocks of $H$, $\mathscr{C}(H)$ is the set of cut-vertices of $H$, and $x\in \mathscr{C}(H)$ is adjacent to $B\in \mathscr{B}(H)$ if and only if $x\in V(B)$. Then $T(H)$ is a tree.

Throughout this subsection, let $(G,C,H)$ be a $k$-triple system with $k\ge 3$.

\begin{lemma}\label{lem:43}
Let $B$ be an end-block of $H$. Suppose that $|V(B)|\ge k\ge 3$ and that $B$ satisfies the following property
\begin{quote}
$(\star)$ for every two vertices $x,y\in V(B)$, there exists an $(x,y)$-path in $B$ of length at least $k-2$.
\end{quote}
Then
$|C|\ge k\bigl(\delta_G(H)-k+2\bigr).$
In particular, the same conclusion holds for $|V(B)|\ge k$ and $c(B)\ge 2(k-2)$.
\end{lemma}

\begin{proof}
Suppose that $H$ is $2$-connected. Set $G_0:=G[H\cup C]$, $B_0:=H$, and choose
an arbitrary vertex $b_0\in V(B_0)$. Otherwise, let $b$ be the cut-vertex
of the end-block $B$. Let $G_0$ be the graph obtained from $G[H\cup C]$
by contracting $V(H)\setminus V(B-b)$ to a single vertex $b_0$, and set
$B_0:=G_0[(B-b)\cup\{b_0\}]$.
In both cases, $B_0\cong B$, and $(G_0,C,B_0)$ is a $k$-triple system:
local $k$-connectivity is inherited after contraction, and for $x,y\in V(C)$, every
$(x,B_0,y)$-path in $G_0$ lifts to an $(x,H,y)$-path in $G$ of at least
the same length.

Let $h:=|V(B_0)|$, and let
$T_0=\{u_1,\dots,u_t\}$ be a maximum strong attachment of $B_0$ to $C$.
By Lemma~\ref{lem:22}, we have $t\ge \min\{k,h\}=k$. Let
$s:=|\NC(B_0)\setminus T_0|$, and define
$r:=\sum_{x\in V(B_0)\setminus\{b_0\}} d_{G_0}(x)/(h-1).$
Since every $x\in V(B_0)\setminus\{b_0\}$ satisfies $d_{G_0}(x)=d_G(x)$, we have $r\geq \delta_G(H)$.
By Theorem~\ref{thm:25}, every strongly attached
pair $\{u_i,u_{i+1}\}$ satisfies
\[
\dstar_{B_0}(u_i,u_{i+1})\ge r+2-t-\frac{s}{h-1}
\ge \delta_G(H)+2-t-\frac{s}{h-1}.
\]
By property $(\star)$, $\dstar_{B_0}(u_i,u_{i+1})\ge k$ for every $i$. Using $h\geq k\geq 3$, we obtain from Lemma~\ref{lem:23}(ii) that
\begin{align*}
|C|
&\ge \sum_{i=1}^t \dstar_{B_0}(u_i,u_{i+1})+2s\ge k\left(\delta_G(H)+2-t-\frac{s}{h-1}\right)+(t-k)k+2s \\
&= k\bigl(\delta_G(H)-k+2\bigr)+s\left(2-\frac{k}{h-1}\right)\ge k\bigl(\delta_G(H)-k+2\bigr).
\end{align*}

For the final assertion, assume $|V(B)|\ge k$ and $c(B)\ge 2(k-2)$. Let $D$ be a cycle in $B$ of length at least $2(k-2)$. Given $x,y\in V(B)$, Menger's theorem provides two internally disjoint $(\{x,y\},D)$-paths with distinct endpoints $a,b\in V(D)$. One of the two $(a,b)$-paths of $D$ has length at least $|D|/2\ge k-2$; concatenating that path with the two $(\{x,y\},D)$-paths gives an $(x,B,y)$-path of length at least $k-2$. Thus, property $(\star)$ holds and the first part applies.
\end{proof}

\begin{lemma}\label{lem:44}
Assume that $H$ is not $2$-connected and that
$|C|<k\bigl(\delta_G(H)-k+2\bigr).$
Let $B$ be an end-block of $H$ with cut-vertex $b$. Then some vertex of $B-b$ has at least $\delta_G(H)-k+2$ neighbors on $C$.
\end{lemma}

\begin{proof}
Suppose not. Then every $x\in V(B-b)$ satisfies
$d_C(x)\le \delta_G(H)-k+1$, and hence $d_B(x)\ge k-1$. In particular,
$|V(B)|\ge k$.

We verify property $(\star)$ of Lemma~\ref{lem:43}. If $|V(B)|=3$ and
$k=3$, then $B$ is a triangle. Otherwise, $|V(B)|\ge4$. Fix distinct
$x_1,x_2\in V(B)$, and choose $z=b$ if $b\notin\{x_1,x_2\}$; otherwise
choose any $z\in V(B)\setminus\{x_1,x_2\}$. Then every vertex of
$B-\{x_1,x_2,z\}$ lies in $B-b$ and has degree at least $k-1$ in $B$.
By Theorem~\ref{thm:bondy-jackson}, there is an $(x_1,x_2)$-path in $B$
of length at least $k-1$. This verifies property $(\star)$, so Lemma~\ref{lem:43} gives
$|C|\ge k(\delta_G(H)-k+2)$, contradicting the assumption.
\end{proof}

\begin{lemma}\label{lem:45}
Assume that $H$ is not $2$-connected. Let $B$ be a non-end-block of $H$. If $c(B)\ge 5k/2$ and $\delta_G(H)\ge 4k$, then
$|C|\ge k\bigl(\delta_G(H)-k+2\bigr).$
\end{lemma}

\begin{proof}
Assume for the contradiction that
\begin{equation}\label{eq:45-contr}
|C|<k\bigl(\delta_G(H)-k+2\bigr).
\end{equation}
View the block structure $T(H)$ as a tree rooted at $B$, where $B$ is a vertex in $T(H)$. For each cut-vertex $x\in V(B)$, let
$T_x$ be the component of $T(H)-B$ containing $x$, and let $V_x\subseteq V(H)$ be the
union of the vertex sets of the blocks in $T_x$. Choose an end-block
$B_x$ in $T_x$, with cut-vertex $b_x$ in $H$. By Lemma~\ref{lem:44}, there is a
vertex $u_x\in V(B_x-b_x)$ such that
$d_C(u_x)\ge \delta_G(H)-k+2.$

Let $G_1$ be obtained from $G[H\cup C]$ by deleting
$V_x\setminus\{x\}$ for every cut-vertex $x\in V(B)$, and then, for every
cut-vertex $x\in V(B)$ and every $c\in N_C(u_x)$, adding the edge $xc$.
Then every non-cut-vertex of $B$ has degree at least $\delta_G(H)$ in
$G_1$, while every cut-vertex of $B$ has at least
$\delta_G(H)-k+2$ neighbors on $C$ in $G_1$.

We show that $(G_1,C,B)$ is a $k$-triple system. Indeed, every cut-vertex
of $B$ has at least $\delta_G(H)-k+2\ge k$ neighbors on $C$ in $G_1$.
For a non-cut-vertex $y$ of $B$, take its original $k$ paths from $y$ to
distinct vertices of $C$ in $G$ that intersect only in $y$. Whenever one
of them first leaves $B$ through a cut-vertex $x$, replace the part after
$x$ by a new edge from $x$ to a vertex of $C$ not used by the other paths.
This is possible since at most one path uses each such $x$, and each $x$
has at least $k$ new neighbors on $C$. This gives the desired $k$ paths
from $y$ to distinct vertices of $C$ in $G_1$. 
The locally longest
property follows because, for every $a,b\in V(C)$, every $(a,B,b)$-path
in $G_1$ lifts to an $(a,H,b)$-path in $G$ of at least the same length.

Let $h:=|V(B)|$, and let $T=\{u_1,\dots,u_t\}$ be a maximum strong
attachment of $B$ to $C$ in $G_1$. By Lemma~\ref{lem:23}(iii), since
$(G_1,C,B)$ is a $k$-triple system and $h\ge c(B)\ge 5k/2$, we have
$t\ge \min\{k,h+|D(T)|\}=k$.
Moreover, the 2-connectivity of $B$ implies that each strongly attached pair satisfies
\begin{equation}\label{eq:45-longpair}
\dstar_B(u_i,u_{i+1})\ge \frac{c(B)}{2}\ge \frac{5k}{4}>k.
\end{equation}

Let $cu(B)$ be the number of cut-vertices of $B$ in $H$.
Since $B$ is not an end-block, $cu(B)\ge 2$.
Set $S:=(N_{G_1}(B)\cap V(C))\setminus T$ and $s:=|S|$.
Recall that every cut-vertex of $B$ has at least $\delta_G(H)-k+2$ neighbors on $C$ in $G_1$, and thus has at least
$\delta_G(H)-k+2-t$ neighbors in $S$. By Lemma~\ref{lem:23}(i), these
neighbor sets are pairwise disjoint. Hence
\begin{equation}\label{eq:45-slower}
s\ge cu(B)\bigl(\delta_G(H)-k+2-t\bigr).
\end{equation}

We show that $cu(B)<k/2$. Suppose to the contrary that
$cu(B)\ge k/2$. Summing~\eqref{eq:45-longpair} over
$i\in\{1,2,\dots,t\}$ gives $|C|>kt$, and hence
\eqref{eq:45-contr} implies
$t<|C|/k<\delta_G(H)-k+2$. Using Lemma~\ref{lem:23}(ii) and then
applying~\eqref{eq:45-slower} together with $cu(B)\ge k/2$, we get
\[
|C|> kt+2s\ge kt+2cu(B)(\delta_G(H)-k+2-t)
\ge kt+k(\delta_G(H)-k+2-t)
= k\bigl(\delta_G(H)-k+2\bigr),
\]
which contradicts~\eqref{eq:45-contr}. This proves $cu(B)<k/2$.

Let $r:=\frac{1}{h}\sum_{x\in V(B)}d_{G_1}(x).$
Since $h\ge c(B)\ge 5k/2$ and $cu(B)<k/2$, we have
\begin{equation}\label{eq:45-r}
r\ge \frac{cu(B)(\delta_G(H)-k+2)+(h-cu(B))\delta_G(H)}{h}
= \delta_G(H)-\frac{(k-2)cu(B)}{h}
\ge \delta_G(H)-\frac{k}{5}.
\end{equation}

We now prove $t< \frac{9k}{5}$. Suppose to the contrary that $t\ge \frac{9k}{5}$. Apply the second statement of Lemma~\ref{lem:23}(iii) to any $k$ of the $t$ strong pairs, and use~\eqref{eq:45-longpair} on the remaining $t-k$ pairs. We have 
\begin{align*}
|C|
&\ge k\left(r+2-t-\frac{s}{h}\right)+(t-k)\frac{5k}{4}+2s\ge k\left(\delta_G(H)-\frac{k}{5}+2-t-\frac{s}{h}\right)+(t-k)\frac{5k}{4}+2s \\
&= k\bigl(\delta_G(H)-k+2\bigr)+\frac{k(5t-9k)}{20}+s\left(2-\frac{k}{h}\right)\ge k\bigl(\delta_G(H)-k+2\bigr),
\end{align*}
which contradicts~\eqref{eq:45-contr}. Hence $t<\frac{9k}{5}.$

Finally, by Lemma~\ref{lem:23}(ii),
$|C|\ge \sum_{i=1}^t \dstar_B(u_i,u_{i+1})+2s$.
Combining Lemma~\ref{lem:23}(iii) with~\eqref{eq:45-r}, every consecutive
strong pair satisfies
\[
\dstar_B(u_i,u_{i+1})\ge r+2-t-\frac{s}{h}
\ge \delta_G(H)-\frac{(k-2)cu(B)}{h}+2-t-\frac{s}{h}.
\]
We use this bound for $k$ pairs, and use the bound~\eqref{eq:45-longpair}
for the remaining $t-k$ pairs. Hence
\begin{align*}
|C|
&\ge k\left(\delta_G(H)-\frac{(k-2)cu(B)}{h}+2-t-\frac{s}{h}\right)
 +(t-k)k+2s \\
&= k\bigl(\delta_G(H)-k+2\bigr)
 +s\left(2-\frac{k}{h}\right)-\frac{k(k-2)cu(B)}{h} \\
&\ge k\bigl(\delta_G(H)-k+2\bigr)
 +\frac{8}{5}cu(B)(\delta_G(H)-k+2-t)
 -\frac{2(k-2)}{5}cu(B) \\
&= k\bigl(\delta_G(H)-k+2\bigr)
 +\frac{cu(B)}{5}\bigl(8\delta_G(H)-10k+20-8t\bigr).
\end{align*}
Since $t<9k/5$ and $\delta_G(H)\ge 4k$, the last extra term is positive,
so $|C|>k(\delta_G(H)-k+2)$, contradicting~\eqref{eq:45-contr} and proving the lemma.
\end{proof}

\begin{theorem}\label{thm:42}
Let $(G,C,H)$ be a $k$-triple system with $k\geq 3$. If $\delta_G(H)\ge 4k$ and some block of $H$ has circumference at least $5k/2$, then
$|C|\ge k\bigl(\delta_G(H)-k+2\bigr).$
\end{theorem}

\begin{proof}
Choose a block $B$ of $H$ with $c(B)\ge 5k/2$. The assertion follows from
Lemma~\ref{lem:43} if $B$ is an end-block, and from Lemma~\ref{lem:45} otherwise.
\end{proof}

\subsection{Completing the proofs}\label{sec:final proof}

We are ready to finish the proof of Theorem~\ref{thm:local-jung}.

\begin{proof}[Proof of Theorem~\ref{thm:local-jung}]
Let $(G,C,H)$ be a $k$-triple system such that $\delta:=\delta_G(H)\ge 6k$ and $p:=p(H)\ge k-1$.
We claim that $|C|\ge k\bigl(\delta-k+2\bigr).$
The case $k=2$ follows from Fan~\cite{Fan1990}, so we may assume $k\ge 3$. 

If $p\le \delta-k+1$, then Theorem~\ref{thm:41} gives the claim. Hence we may assume that
$p\ge \delta-k+2\ge 5k+2.$
If some block of $H$ has circumference at least $5k/2$, then Theorem~\ref{thm:42} gives the desired inequality. Therefore, every block of $H$ has circumference less than $5k/2$, and hence $c(H)<5k/2$.

Now we apply Lemma~\ref{lem:31}(ii) to obtain a set $U=\{u_1,\dots,u_p\}$ together with pairwise internally disjoint linking paths as in Proposition~\ref{prop:32}, such that
$\sum_{i=1}^p d_H(u_i)\le (2c(H)-2)(p-1)\le (5k-2)(p-1).$
Since every $u_i$ has degree at least $\delta$ in $G$, it holds that
\begin{align*}
e(U,C)
&= \sum_{i=1}^p d_C(u_i)
=\sum_{i=1}^p \bigl(d_G(u_i)-d_H(u_i)\bigr)\ge p\delta-(5k-2)(p-1) \\
&= p(\delta-5k+2)+5k-2\ge (5k+2)(\delta-5k+2)+5k-2.
\end{align*}

Since each vertex of $C$ has at most $p$ neighbors in $U$ while $e(U,C)\ge p(\delta-5k+2)+5k-2>p$, the neighbor set $N$ of $U$ on $C$ has size at least $2$. 
Applying~\eqref{eq:41-counting} to $U$ gives
$|C|\ge e(U,C)+|N_C(U)|\ge e(U,C)$. 
We assert that the above lower bound for $e(U,C)$ is at least $k(\delta-k+2)$ whenever $\delta\ge 6k$; this is because $(5k+2)(\delta-5k+2)+5k-2-k(\delta-k+2)=(4k+2)\delta-24k^2+3k+2>0.$
This proves $|C|\ge k(\delta-k+2)$.
\end{proof}

Finally, we prove Theorem~\ref{thm:local-bondy}.

\begin{proof}[Proof of Theorem~\ref{thm:local-bondy}]
Let $(G,C,H)$ be a $k$-triple system. 
Assume for contradiction that $p(H)\ge k$. Let $c:=|C|$. Since
$\delta_G(H)\ge \frac{n+k(k-1)}{k+1}\ge 6k$
for $ n\ge 5k^2+7k,$
Theorem~\ref{thm:local-jung} yields
$c\ge k\left(\frac{n+k(k-1)}{k+1}+2-k\right)=\frac{k(n+2)}{k+1}.$
Hence
$|V(G-C)|=n-c\le n-\frac{k(n+2)}{k+1}=\frac{n-2k}{k+1}.$
Now every vertex $v\in V(H)$ has at most $|V(H)|-1\le |V(G-C)|-1$ neighbors in $H$, so
$d_C(v)\ge d_G(v)-d_H(v)\ge \frac{n+k(k-1)}{k+1}-\frac{n-2k}{k+1}+1=k+1.$
Thus, $H$ is locally $(k+1)$-connected to $C$. 
Since $\delta_G(H)-k\ge \frac{n+k(k-1)}{k+1}-k=\frac{n-2k}{k+1}\ge |V(G-C)|,$
we have
$k\le p(H)\le |V(G-C)|\le \delta_G(H)-k.$
Applying Theorem~\ref{thm:41} (with $k+1$ in place of $k$ therein) gives
$c\ge (k+1)\bigl(\delta_G(H)-k+1\bigr)
\ge (k+1)\left(\frac{n+k(k-1)}{k+1}-k+1\right)=n-k+1.$
But we have $|V(H)|\ge p(H)\ge k$, and therefore
$n=|V(G)|\ge c+|V(H)|\ge (n-k+1)+k=n+1,$
a contradiction. This completes the proof of Theorem~\ref{thm:local-bondy}.
\end{proof}

\section{Concluding remarks}\label{sec:conclude}
In this paper, we prove Bondy's conjecture for all sufficiently large graphs.
We conclude with a discussion of related conjectures and several new problems motivated by our results. 

\subsection{On a conjecture of Fan}
Fan~\cite{Fan1990} proposed the following related conjecture in 1990, stated here in our notation: 
if $(G,C,H)$ is a $k$-triple system such that $|V(H)|\ge k-1$ and the average degree of $H$ in $G$ is $r$, 
then $|C|\ge k(r-k+2).$
The conjecture was proved by Fan~\cite{Fan1990} for $2\le k\le 4$, but it is false for every $k\ge 5$ (see \cite{Nagayama2003}).

The methods developed in this paper allow us to establish the conjectured conclusion under the additional assumptions that $H$ is $2$-connected and relatively dense. 
This may be viewed as supporting evidence for Jung's conjecture in the dense regime (i.e., when $H$ is dense). For a graph $H$, let $d(H)$ denote its average degree.

\begin{theorem}\label{lem:dense-H}
Let $(G,C,H)$ be a $k$-triple system.
Define $r=\sum_{v\in V(H)}d_G(v)/|V(H)|.$
If $H$ is 2-connected with $d(H)\geq k-2$ and $r\geq 2k-1$,\footnote{The condition $r\geq 2k-1$ can be replaced with $3\leq r\leq 2k-4$; see the remark at the end of this subsection.} then $|C|\geq k(r-k+2)$.
\end{theorem}

\begin{proof}
Let $T=\{u_1,u_2,\ldots,u_t\}$ be a maximum strong attachment of $H$ to $C$,  $S=N_C(H)\backslash T$ and $s=|S|$.
Let $h=|V(H)|$. Then $h\geq d(H) +1\geq k-1$.
By the second statement in Lemma~\ref{lem:23}(iii), every strongly attached pair $\{u_i,u_{i+1}\}$ satisfies
\begin{align}\label{equ:d*}
d^{*}_H(u_i,u_{i+1})\geq r+2-t-\frac{s}{h}
\end{align}
and by Lemma~\ref{lem:23}(ii),
\begin{align}\label{equ:C-1}
|C|\geq {\textstyle\sum_{i=1}^t} d^{*}_H(u_i,u_{i+1})+2s.
\end{align}
By Lemma~\ref{lem:24}, for every strongly attached pair $\{u_i,u_{i+1}\}$,
$d^*_H(u_i,u_{i+1})\ge d(H)+\frac{e(\{u_i,u_{i+1}\},H)}{h}.$
Summing this inequality over all consecutive pairs of $T$, and observing that every edge between $T$ and $H$ is counted twice, gives
\begin{align}\label{equ:sum-d*}
{\textstyle\sum_{i=1}^t} d^{*}_H(u_i,u_{i+1})\geq t\cdot d(H)+\frac{2e(H,T)}{h}.
\end{align}

Let $D(T)=\{v\in T: d_H(v)\geq 2\}$.
By Lemma~\ref{lem:23}(iii), we have $t\geq \min\{k,h+|D(T)|\}$.
Hence, either $t\geq k$, or $t=h=k-1$ and $D(T)=\emptyset$.
Suppose that the latter case occurs.
Then $h-1\geq d(H)\geq k-2$ forces $d(H)=k-2$ and thus $H$ is a clique of size $k-1$.
Since $D(T)=\emptyset$, every vertex in $N_C(H)$ has exactly one neighbor in $H$.
This implies that $s+t=e(H,C)=\sum_{v\in V(H)}(d_G(v)-d_H(v))
=h(r-(h-1))=(k-1)(r-k+2)$, and thus $s=(r-k+1)(k-1)$.
For each strongly attached pair $\{u_j,u_{j+1}\}$, choose distinct 
$a\in N_H(u_j)$ and $b\in N_H(u_{j+1})$. Since $H\cong K_{k-1}$, there is an $(a,b)$-path of length $k-2$ in $H$, and hence a $(u_j,H,u_{j+1})$-path of length $k$.
The local maximality of $C$ then gives $|C[u_j,u_{j+1}]|\geq \dstar_H(u_j,u_{j+1})\ge k$.
Thus, as $k\geq 2$, by Lemma~\ref{lem:23}(ii),
$|C|\geq (k-1)k+2s=(k-1)k+2(k-1)(r-k+1)\geq k(r-k+2).$
We may therefore assume that $t\geq k$.

Suppose that either $d(H)\geq k-1$ or $k\leq 4$ holds.
By Lemma \ref{lem:24}, for any strongly attached pair $\{u_i,u_{i+1}\}$
we have $d^{*}_H(u_i,u_{i+1})\geq d(H)+\frac{e(\{u_i,u_{i+1}\},H)}{h}>d(H)$; also, since $H$ is 2-connected, $d^{*}_H(u_i,u_{i+1})\geq 4$.
This shows that in either case, $d^{*}_H(u_i,u_{i+1})\geq k$.
By \eqref{equ:d*} and \eqref{equ:C-1}, we obtain
$|C|\geq k(r-t+2-\frac{s}{h})+(t-k)k+2s=k(r-k+2)+s(2-\frac{k}{h})\geq k(r-k+2).$
We may therefore assume that $k-2\leq d(H)< k-1$ and $k\geq 5$.

In the rest of the proof, we distinguish between two cases.
First, consider the case $r\geq d(H)+t$.
Averaging~\eqref{equ:sum-d*}, there exist $(t-k)$ strongly attached pairs, without loss of generality, say
$\{u_i,u_{i+1}\}_{1\leq i\leq t-k}$, such that
\begin{align}
\sum_{i=1}^{t-k} d^{*}_{H}(u_i,u_{i+1})
&\geq (t-k)\cdot d(H)+\frac{2(t-k)}{t}\cdot \frac{e(H,T)}{h}\nonumber= (t-k)\cdot d(H)+\frac{2(t-k)}{t}\cdot \frac{e(H,C)-s}{h}\nonumber\\
&=(t-k)\cdot d(H)+\frac{2(t-k)}{t}\cdot (r-d(H))-\frac{2s(t-k)}{th}.\nonumber
\end{align}
Since $r-d(H)\geq t\geq k$ and $d(H)\geq k-2$, we obtain that
$\sum_{i=1}^{t-k}d^{*}_{H}(u_i,u_{i+1})\geq (t-k)(d(H)+2)-\frac{2s}{h}\geq (t-k)k-\frac{2s}{h}.$
Using \eqref{equ:d*}, \eqref{equ:C-1}, and the fact $h\geq k-1$ and $k\geq 5$, we derive the desired lower bound
\begin{align*}
|C|&\geq \sum_{i=1}^{t-k}d^{*}_{H}(u_i,u_{i+1})+\sum_{i=t-k+1}^{t}d^{*}_{H}(u_i,u_{i+1})+2s\geq \left((t-k)k-\frac{2s}{h}\right)+k(r-t+2-\frac{s}{h})+2s\\
&= k(r-k+2)+s\left(2-\frac{2}{h}-\frac{k}{h}\right)\geq k(r-k+2).
\end{align*}

It remains to consider the case $r<d(H)+t$. That is, $t>r-d(H)$. By \eqref{equ:C-1} and \eqref{equ:sum-d*},
\begin{align*}
|C|&\geq t\cdot d(H)+\frac{2e(H,T)}{h}+2s=t\cdot d(H)+\frac{2(e(H,C)-s)}{h}+2s\\
&=t\cdot d(H)+2(r-d(H))-\frac{2s}{h}+2s=(t-2)\cdot d(H)+2r+2s-\frac{2s}{h}\\
&\geq (t-2)\cdot d(H)+2r\geq (r-d(H)-2)\cdot d(H)+2r.
\end{align*}
Set $f(x):=(r-x-2)x+2r$, where $x:=d(H)\in [k-2,k-1]$.
It is concave and symmetric about $x=\frac{r-2}{2}$.
Since $r\geq 2k-1$, $(k-2)+(k-1)\leq r-2$ and thus $f(k-1)\geq f(k-2)$.
Therefore,
$|C|\geq f(x)\geq \min\{f(k-2),f(k-1)\}=f(k-2)=(r-k)(k-2)+2r=k(r-k+2),$
completing the proof.
\end{proof}

We remark that Theorem~\ref{lem:dense-H} remains true if the condition $r\geq 2k-1$ is replaced with $3\le r\le 2k-4$. To see this, set $k':=\lfloor (r+1)/2\rfloor\ge 2$. Then $r\ge 2k'-1$ and
$2\le k'\le k$. Since a $k$-triple system is also a $k'$-triple system, and $d(H)\ge k-2\ge k'-2$, we can apply Theorem~\ref{lem:dense-H} (with respect to $k'$) to derive $|C|\ge k'(r-k'+2).$
Let $g(x)=x(r+2-x)$. Since $g$ is concave and symmetric about
$(r+2)/2$, and since $k+k'\ge r+2$ (this is because $r\le 2k-4$), we have $g(k')\ge g(k)$. Therefore $|C|\ge g(k')\ge g(k)=k(r-k+2).$

\subsection{On a conjecture of Voss}

Following Veldman~\cite{Veldman1983}, for an integer $\ell\ge 1$ we say that a cycle $C$ in a graph $G$ is a {\it $D_\ell$-cycle} if every component of $G-C$ has order less than $\ell$. 
Voss~\cite[p.~151]{Voss1991} conjectured that if $G$ is $k$-connected and $\delta(G)\ge r\ge 2k$, then every longest cycle of $G$ is either a $D_{k-1}$-cycle or has length at least $k(r-k+2)$. 

This conjecture was proved for $2\le k\le 4$ (see~\cite[p.~151]{Voss1991}). In particular, the case $k=2$ corresponds to Dirac's long cycle theorem~\cite{Dirac1952}. 
The next result shows that the conjecture fails for all $k\ge 6$.

\begin{theorem}
For every integer $k\geq 6$,
there exist an integer $r\geq 2k$ and a $k$-connected graph $G$ with minimum degree at least $r$, circumference smaller
than $k(r-k+2)$, and a longest cycle $C$ of $G$ that is not a $D_{k-1}$-cycle.
\end{theorem}

\begin{proof} 
Let \(A\) be a clique of order \(2k\), let \(B\) be a copy of
the star \(K_{1,2k}\), and let \(T_1,\ldots,T_{2k}\) be pairwise disjoint
triangles.  Define
$G=A\vee \bigl(B\cup T_1\cup\cdots\cup T_{2k}\bigr),$
that is, every vertex of \(A\) is joined to every vertex of
\(B\cup T_1\cup\cdots\cup T_{2k}\), and there are no other edges between
distinct members of \(B,T_1,\ldots,T_{2k}\).

We first compute the relevant parameters.  Let $\kappa(G)$ denote the connectivity of $G$.
Since deleting \(A\) disconnects
the graph, \(\kappa(G)\le 2k\).  On the other hand, after deleting fewer than
\(2k\) vertices, at least one vertex of \(A\) remains, and this vertex is
adjacent to every other remaining vertex.  Hence \(G\) remains connected.
Thus
$\kappa(G)=2k,$ and
so \(G\) is \(k\)-connected.  Moreover, a leaf of \(B\) has degree \(2k+1\),
while every other vertex has degree at least \(2k+1\).  Therefore,
$\delta(G)=2k+1.$
Set \(r=2k+1\).  Then \(r\ge 2k\).

We next determine the circumference of \(G\).  Let \(C\) be any cycle of
\(G\).  The vertices of \(C\) outside \(A\) occur in path segments lying in
the components \(B,T_1,\ldots,T_{2k}\) of \(G-A\).  There are at most
\(|A|=2k\) such segments, because each non-empty segment must be inserted
between two consecutive vertices of \(C\cap A\).  Each such segment has at
most three vertices, since the longest path in \(K_{1,2k}\) has order \(3\),
and the longest path in a triangle also has order \(3\).  Hence
$|C|\le |A|+3|A|=2k+6k=8k.$
This bound is attained by a cycle using all vertices of \(A\) and all
vertices of the triangles \(T_1,\ldots,T_{2k}\), inserting each triangle as a
three-vertex path between two consecutive vertices of \(A\).  Consequently
$c(G)=8k.$
Since \(k\ge 6\) and $r=2k+1$, we have
$k(r-k+2)=k(k+3)>8k=c(G)$.

Let \(C\) be such a longest cycle.  Then \(C\) avoids all vertices of the
star \(B\).  Thus, \(B\) is a component of \(G-V(C)\), and
$|V(B)|=2k+1\ge k-1.$
By the definition of a \(D_{k-1}\)-cycle used above, \(C\) is not a
\(D_{k-1}\)-cycle.
Thus, \(C\) is a longest cycle of a \(k\)-connected graph with
\(\delta(G)\ge r\ge 2k\), but \(C\) is not a \(D_{k-1}\)-cycle and
$|C|=8k<k(r-k+2).$
This gives the desired counterexample for Voss's conjecture.
\end{proof}

\subsection{Two new problems}
Motivated by the results in this paper, it is natural to consider the following problems.
Recall that $p(H)$ denotes the number of vertices of a longest path in a graph $H$. 

\begin{problem}\label{prob:two}
Let $k\geq 2$.
Let $C$ be a cycle of length $c$ in a 2-connected graph $G$, and let $H$ be a component of $G-C$. Suppose that $C$ is locally longest with respect to $H$ and that $H$ is locally $k$-connected to $C$. If $p(H)\geq k-1$, then
$c\ge k(\delta_G(H)-k+2).$
\end{problem}

\begin{problem}\label{prob:three}
Let $k\geq 2$.
Let $G$ be a $2$-connected graph on $n$ vertices and let $C$ be a longest cycle of $G$. Let $H$ be a component of $G-C$ that is locally $k$-connected to $C$. If
$\delta(G)\ge \frac{n+k(k-1)}{k+1},$
then $p(H)\leq k-1$.
\end{problem}

We conclude this paper with the following implications. 

\begin{proposition}\label{prop:51}
Problem~\ref{prob:two} implies Problem~\ref{prob:three} and Jung's conjecture.
Moreover, Problem~\ref{prob:three} implies Conjecture~\ref{conj:bondy}.
\end{proposition}

\begin{proof}
We show that Problem~\ref{prob:two} implies Problem~\ref{prob:three}, while the other implications follow directly.
Suppose to the contrary that $p(H)\ge k$. Since $p(H)\ge k-1$ and $H$ is locally $k$-connected to $C$, a solution of Problem~\ref{prob:two} gives
$|C|\ge k\left(\frac{n+k(k-1)}{k+1}+2-k\right)=\frac{k(n+2)}{k+1}.$
Therefore,
$|V(G-C)|=n-|C|\le n-\frac{k(n+2)}{k+1}=\frac{n-2k}{k+1}.$
For every $v\in V(H)$,
$d_C(v)\ge d_G(v)-d_H(v)\ge \frac{n+k(k-1)}{k+1}-\frac{n-2k}{k+1}+1=k+1.$
Hence $H$ is locally $(k+1)$-connected to $C$. Applying Problem~\ref{prob:two} again, now with $k+1$ in place of $k$, we obtain
$|C|\ge (k+1)\left(\frac{n+k(k-1)}{k+1}+2-(k+1)\right)=n-k+1.$
But $p(H)\ge k$ implies $|V(H)|\ge k$, so
$|C|=n-|V(G-C)|\le n-|V(H)|\le n-k,$
a contradiction. 
This gives a solution to Problem~\ref{prob:three}.
\end{proof}

\section*{Acknowledgment}
The first and second authors are grateful to Binlong Li and Hehui Wu for many helpful discussions over the years concerning parts of an earlier version of this manuscript.

{\fontsize{9.5}{14}\selectfont
}
\end{document}